\documentclass{article}
\usepackage{latexsym, amsfonts, amsmath, amssymb, amscd, dsfont, epsfig}
\usepackage[paperwidth=7.5in, paperheight=10.5in, margin=.875in]{geometry}
\usepackage[colorlinks,linkcolor=red,anchorcolor=green,citecolor=blue]{hyperref}
\hypersetup{colorlinks=true,linkcolor=black}
\usepackage{amsfonts,amssymb}
\usepackage{amsmath}
\usepackage{graphicx,color}
\usepackage{cite}
\usepackage{comment}
\usepackage{enumerate}
\definecolor{asparagus}{rgb}{0.53, 0.66, 0.42}
\providecommand{\U}[1]{\protect\rule{.1in}{.1in}}
\newtheorem{theorem}{Theorem}

\newtheorem{definition}[theorem]{Definition}

\newtheorem{lemma}[theorem]{Lemma}

\newtheorem{proposition}[theorem]{Proposition}

\newenvironment{proof}[1][Proof]{\textit{#1.} }{\ \rule{0.5em}{0.5em}}

\begin{document}

\title{Extension of Hamel paradox for the 2D exterior Navier-Stokes problem}
\author{Zhengguang Guo\footnote{School of Mathematics and Statistics, Huaiyin Normal  University, Huaian 223300, Jiangsu, China; E-mail: gzgmath@hytc.edu.cn } \& Matthieu Hillairet\footnote{Institut Montpelliérain Alexander Grothendieck, Univ. Montpellier, CNRS, Montpellier, France; E-mail: matthieu.hillairet@umontpellier.fr
}}
\maketitle
\begin{abstract}
In this paper, we continue the analysis of the stationary exterior Navier-Stokes problem with interior boundary data and vanishing condition at infinity. We first show an existence result that extends a previous contribution of the second author \cite{HillairetWittwer} by considering boundary data prescribing a non-trivial flux on the internal boundary. We obtain in particular that the non-uniqueness result of G. Hamel \cite{Hamel} extends to an open set of internal boundary data.  We then show that one way to recover uniqueness of a solution is to complement the perturbation of velocity field with a decay condition at infinity for small circulation through the interior boundary. Our method is based on a fine analysis of the linearized Navier-Stokes system around potential flows in the exterior domain.
\end{abstract}

\tableofcontents

\section{Introduction}

In this paper we address the existence and uniqueness of solutions to the planar Navier-Stokes system in the exterior of a unit disk $B\subset \mathbb{R}^2$: 
\begin{equation} \label{eq_NS2D}
\left\{
\begin{aligned}
\left( {\bf u} \cdot \nabla\right) {\bf u} + \nabla p & =  \Delta {\bf u}, && \text{in $\mathbb R^2 \setminus  \overline{B}$},\\
{\rm div} {\bf u} & = 0,&& \text{in $\mathbb R^2 \setminus  \overline{B}$}, \\
{\bf u} &= {\bf u}^*, && \text{on $\partial B$}, \\
{\bf u} &= 0, && \text{at  infinity}.
\end{aligned}
\right.
\end{equation} 
This system represents the stationary motion of a viscous incompressible fluid outside $B$ where an inflow ${\bf u}^*$ is imposed.  It has been widely studied for years. Following the seminal approach of J. Leray \cite{leray} a first method to solve this system consists in constructing solutions with finite $\dot{H}^1$-norm by an exhaustion method \cite{fujita}. Unfortunately, this weak solution method is irrelevant in this 2D exterior case since it does not ensure the boundary condition at infinity is satisfied \cite{KorPiRu}. More references on this topic can be found in the reference book on exterior problems \cite{Galdi}. Particularly, the existence and uniqueness for a non-zero boundary condition at infinity was discussed in \cite[Section XII.5]{Galdi} or a recent work \cite{GazzolaNeustupaSperone}.
One way to explain the flaw of the exhaustion method is that the standard linearization at infinity around the profile $({\bf u}, p)=(0, 0)$ is not uniquely solvable. This observation is named ‘‘Stokes paradox" after the seminal computations of G.G.  Stokes \cite[Section V]{Galdi}. To circumvent this difficulty, a straightforward idea is to look for an {\em a priori} first order asymptotics of a possible solution in order to perform a perturbation method with a more suitable linearized system. A natural candidate in this respect is a potential flow $\bf u$ that cancels independently the two terms $({\bf u} \cdot \nabla) {\bf u} + \nabla p$ and $\Delta {\bf u}$ with $p = -|{\bf u}|^2/2$. This idea is proposed and implemented when $B$ is a disk in \cite{HillairetWittwer} with specific $-1$-homogeneous potential flows corresponding to rotating profiles with rotation velocity decaying at  infinity. However, the Stokes paradox still influences the computations in this setting. It implies that if one chooses a small perturbation of a given profile for boundary data on $\partial B,$ the associated solution will converge to another profile at infinity. This result is limited to sufficiently fast rotation velocities (on $\partial B$). The method of proof relies on a sharp description of the associated linearized system with a Fourier series argument based on the symmetries of the geometry. The optimality of the limiting condition on the rotation velocity is supported by numerical experiment \cite{Guiwer}. Other candidates for asymptotic profiles are computed in \cite{Guillod}.  The method and results of \cite{HillairetWittwer} have been improved and extended to various directions: computations of stationary solutions in a rotating frame \cite{GalHiMa}, other factorization methods \cite{MaTsu}, boundary data with flux \cite{Higaki}, three-dimensional setting \cite{HiHo}, time-dependent problem \cite{EHL}, etc. This review is restricted to general boundary data with no symmetry assumptions. We also note that the general existence theory for (\ref{eq_NS2D}), even for small data, is still open, and few results are available only under symmetry assumptions on the data (see  \cite{Galdiarticle} for instance).

A further obstacle to the possibility of an existence/uniqueness result for \eqref{eq_NS2D} is due to G. Hamel. In his seminal contribution \cite{Hamel}, G. Hamel constructed a three-parameter family of solutions to \eqref{eq_NS2D} with remarkable features,  see also \cite[p. 803]{Galdi}. To explain his result in more details, let consider $B = B(0,1)$ is a unit disk and introduce $(r,\theta)$ the corresponding polar coordinates (with the associated local basis $({\bf e}_r,{\bf e}_{\theta})$).  We can then split any vector-field ${\bf v}: \mathbb R^2 \setminus B(0,1) \to \mathbb R^2$ into:
\[
{\bf v} = v_r {\bf e}_r + v_{\theta}{\bf e}_{\theta}. 
\]
The velocity-field ${\bf v}^{(H)}$ of a Hamel flow is fixed by the three constants $(\phi,\mu,\lambda) \in \mathbb R^3$ and given by the formulas:
\begin{equation} \label{eq_Hamelflow}
v^{(H)}_r(r,\theta) = -  \dfrac{\phi}{r} \qquad v^{(H)}_{\theta}(r,\theta)
 = \lambda r^{1-\phi} + \dfrac{\mu}{r} . 
\end{equation}
We recognize a flow which consists of a constant suction proportional to $\phi$ (ensuring that the flux through circles is constant) and a rotation velocity depending on $r$ (thus the name of spiral flow).  These velocity-fields vanish at infinity in case 
$\lambda=0$ or $\phi >  1$.  Considering these flows outside $B,$ we observe that the restriction of ${\bf v}^{(H)}$ on $\partial B$ prescribes $\phi$ and thus completely  the radial component ${v}^{(H)}_r$ of ${\bf v}^{(H)}$ on $\partial B.$ On the contrary, the tangential part of ${\bf v}^{(H)}$ being a combination of two parts, prescribing the boundary values leaves one free parameter. For instance, prescribing vanishing boundary tangential velocity enforces only that $\lambda + \mu =0$ and we loose uniqueness. In this paper, we build upon the following further remarks. If we consider that $\phi  > 2,$  the leading term of $v^{(H)}_{\theta}$ at infinity is $\mu/r$ that entails two properties:
\begin{itemize}
\item[i)] when $r >> 1,$ ${\bf{v}}^{(H)}$ is a perturbation of :
\begin{equation} \label{eq_uref}
{\bf u}_{ref}[\phi,\mu](r,\theta) = \dfrac{-\phi}{r} {\bf e}_r + \dfrac{\mu}{r} {\bf e}_{\theta}.
\end{equation}
\item[ii)] prescribing $\lim_{r \to \infty} v_{\theta}^{(H)}(r,\theta)r$ fixes the constant $\mu$ and thus the  combination in \eqref{eq_Hamelflow}. 
\end{itemize}

For the analysis below, we recall that 
\[
2 \pi \phi = \int_{\partial B} {\bf v}^{(H)} \cdot {\bf n} {\rm d}\sigma
\]
is the flux prescribed by the boundary value for ${\bf v}^{(H)}.$ We recall that this is an invariant since ${\bf v}^{(H)}$ is divergence-free. In particular, a similar identity holds when replacing $\partial B$ by any simple curve circling around $B$ (a circle centered in the origin of radius larger than $1$ for example). Following the theory on inviscid flows \cite{MarchioroPulvirenti}, the quantity
\[
2\pi \mu =  \int_{\partial B} {\bf v}^{(H)} \cdot {\bf n}^{\bot} {\rm d}\sigma
\]
is the circulation prescribed by the boundary value of ${\bf v}^{(H)}$ on $\partial B.$ 

To state our main result, we decompose any boundary data ${\bf u}^{*}$ involved in \eqref{eq_NS2D} as follows:
\[
{\bf u}^* =  - \phi_0 {\bf e}_r + \mu_0 {\bf e}_{\theta} +  {\bf v}^*  \qquad 
\]
with ${\bf v}^* \in C^{\infty}(\partial B)$ a small perturbation prescribing no flux nor circulation through $\partial B.$  
Our main existence result then reads
\begin{theorem} \label{thm_existence_intro}
Let $\phi_0 \in [0,\infty)$ and $\mu_0 \in \mathbb R$ satisfy:
\begin{equation} \label{eq_conditionphimu_intro}
\phi_0 > \frac 32 \qquad \text{ or } \qquad \left( \phi_0 \in \left[0, 3/2 \right] \; \text{ and }\; \; |\mu_0| > (4-\phi_0) \sqrt{3-2\phi_0} \right).  
\end{equation}
There exists a ball of perturbation ${\bf v}^*$ (for some topology to be made precise) satisfying
\[
\int_{\partial B}{\bf v}^* \cdot {\bf n}{\rm d}\sigma = 0, \qquad \int_{\partial B} {\bf v}^{*} \cdot {\bf n}^{\bot} {\rm d}\sigma = 0
\] 
such that:
\begin{itemize}
\item[if] $0\leq \phi_0 \leq 2,$ there exists a smooth solution $({\bf u},p)$ to \eqref{eq_NS2D} with boundary data 
\[
{\bf u}^* = - \phi_0 {\bf e}_r + \mu_0 {\bf e}_{\theta} +  {\bf v}^*.
\]
Moreover, there exists $\mu$ close to $\mu_0$ such that:
\[
\lim_{r \to \infty} \sup_{\theta \in (-\pi,\pi)} r \left| {\bf u}(r,\theta) - {\bf u}_{ref}[\phi_0,\mu](r,\theta) \right| = 0.
\]
\item[if] $\phi_0 > 2,$  for every $\mu$ close to $\mu_0$, there exists a smooth solution $({\bf u}_{\mu},p_{\mu})$ to \eqref{eq_NS2D} with boundary data 
\[
{\bf u}^* = - \phi_0 {\bf e}_r + \mu_0 {\bf e}_{\theta} + {\bf v}^*,
\]
such that:
\[
\lim_{r \to \infty} \sup_{\theta \in (-\pi,\pi)} r \left| {\bf u}_{\mu}(r,\theta) - {\bf u}_{ref}[\phi_0,\mu](r,\theta) \right| = 0.
\]
\end{itemize}
\end{theorem}

The condition \eqref{eq_conditionphimu_intro} is reminiscent of the sufficiently fast rotation assumption introduced in \cite{HillairetWittwer}. It is based on the computation of the linearized \eqref{eq_NS2D} around a profile $({\bf u}_{ref}[\phi_0,\mu], p_{ref}[\phi_0,\mu] = -|{\bf u}_{ref}[\phi_0,\mu]|^2/2)$ (note that $\mu$ is not necessarily $\mu_0$).  As in \cite{HillairetWittwer}, this assumption shall ensure that Green functions associated with the linearized problem decay sufficiently fast to be combined with the nonlinearity.  We emphasize that, in case $\phi_0=0$ we recover
the condition
$
|\mu_0| > 4\sqrt{3}
$
which was the condition derived in\cite{HillairetWittwer}. The main originality of this existence result is the case $\phi_0 > 2$ where we obtain a one-parameter family of solutions for many given boundary data. Briefly, this innovation is permitted because we do not face the non-invertibility of the Stokes problem when we consider the linearized Navier-Stokes system around $({\bf u}_{ref}[\phi_0,\mu], p_{ref}[\phi_0,\mu])$ with $\phi_0$ sufficiently large. We can then implement a classical perturbation analysis and obtain existence of a solution for boundary data close to ${\bf u}_{ref}[\phi_0,\mu]$. Furthermore, we can play with the parameter $\mu:$ the same boundary data ${\bf u}^*$ can read ${\bf u}_{ref}[\phi_0,\mu] + {\bf v}$ for various values of $\mu$ (close to $\mu_0$). This is possible since the mapping $\mu \mapsto {\bf u}_{ref}[\phi_0,\mu] $ is continuous in any $H^m(\partial B)$-space. In this respect, our result is a non-trivial extension of the existence result in \cite{Higaki} where this non-uniqueness property was only mentioned for the spiral solutions of Hamel. We obtain herein that this non-uniqueness property is generic. The key role of the circulation parameter $\mu$ has already been identified in the Cauchy theory for unbounded-energy solutions to the time-dependent Navier-Stokes equations in exterior domains \cite{FeHi,GalMa,Lacave}.  

Once we have a one-parameter family of solutions, a natural issue is to find a criterion that enables to discriminate between these solutions. Following the remark ii) above, a natural candidate is to prescribe 
$
\lim_{r \to \infty} r u_{\theta}(r,\theta)
$ 
or equivalently to prescribe the value of $\mu$ of the leading order ${\bf u}_{ref}[\phi_0,\mu]$ of the solution. Our main result is the following:
\begin{theorem} \label{thm_uniqueness}
	Let $\phi_0 \in (21/10,3]$ and $\beta_0>0.$ There exists $\varepsilon(\phi_0,\beta_0) > 0$  such that the following statement holds true. 
	
	If $\mu \in \mathbb R$ and $({\bf u}_i,p_i) \in C^{\infty}(\mathbb R^2\setminus \overline{B}) \times C^{\infty}(\mathbb R^2 \setminus \overline{B})$ for $i=a,b$ are solutions to \eqref{eq_NS2D} with the same ${\bf u}^{*} \in C^{\infty}(\partial B)$ and such that ${\bf v}_i = {\bf u}_i - {\bf u}_{ref}[\phi_0,\mu]$ satisfies:
		\begin{align} \notag 
			 \sup_{|x| \geq 1} \left[ |x|^{1+\beta_0} (|{\bf
					v}_a(x)|+|{\bf v}_b(x)| \right.  & + |x| (|\nabla {\bf v}_a(x)| + |\nabla {\bf v}_{b}(x)|) 				\\
					& \left. + 
					|x|^{2} (|\nabla^2 {\bf v}_a(x)| + |\nabla^2 {\bf v}_{b}(x)|) \right] + |\mu| < \varepsilon(\phi_0,\beta_0),
		\label{eq_unicite}	
		\end{align}
then ${\bf u}_{b} = {\bf u}_a.$
\end{theorem}

In practice, $\mathbf{u}_a$ is a solution such as the one we construct in the existence part. The second solution $\mathbf{u}_b$ can be any solution. In the above theorem, we require this solution to be smooth, but it is classical that weak solutions are as smooth as the boundary data thanks to the ellipticity of the Stokes system. So this is no restriction when boundary data are smooth. Assumption \eqref{eq_unicite} means that the perturbations ${\bf v}_a$ and ${\bf v}_b$ decay faster than $1/|x|$ at infinity.  This is mandatory since otherwise our assumptions would allow to compare solutions that we constructed with asymptotic profiles ${\bf u}_{ref}[\phi_0,\tilde{\mu}]$ with different values for $\tilde{\mu}.$ In contrast with the uniqueness result in \cite{Higaki}, one strength of our result is furthermore that we allow the faster decay of the perturbation to be arbitrary small (we have no size-restriction on $\beta_0$).

The proof of {\bf Theorem \ref{thm_uniqueness}} is mainly based on an energy method. However, several difficulties pave the way to the result. Indeed, let fix a pair of solutions $({\bf u}_a,p_a)$ and $({\bf u}_b,p_b)$ and introduce
\begin{equation*}
	\mathbf{w}=\mathbf{u}_{a}-\mathbf{u}_{b},\qquad q=p_{a}-p_{b}.
\end{equation*}%
The difference satisfies
\begin{eqnarray}	\label{w equintro}
	\left( \mathbf{w}\cdot \nabla \right) \mathbf{u}_{a}+\left( \mathbf{u}%
	_{b}\cdot \nabla \right) \mathbf{w}+\nabla q &=&\Delta \mathbf{w,}
 \\
	\text{div }\mathbf{w} &\mathbf{=}&0,  \notag
\end{eqnarray}
with vanishing boundary conditions on $\partial B.$ If one multiplies this system directly with ${\bf w},$ it yields after integration by parts:
\[
\int_{\mathbb R^2 \setminus \overline{B}} |\nabla {\bf w}|^2 {\rm d}x
= -\int_{\mathbb R^2 \setminus \overline{B}} {\bf w} \cdot \nabla {\bf u}_a \cdot {\bf w}{\rm d}x.
\]
Because of the leading term in ${\bf u}_a$ the right-hand side contains a priori a term like
\[
\int_{\mathbb R^2 \setminus \overline{B}} \left( \phi_0 + \mu \right) \dfrac{|{\bf w}|^2}{|x|^2}{\rm d}x.
\]
However, on the one hand, we have no Hardy inequality in exterior $2D$ domains that would allow to control this term with the left-hand side. On the other hand, there is no hope to make this term small since $\phi_0$ is not arbitrary small in our setting. 
Hence, we need to go into more details when computing this right-hand side. Firstly, we will extract the terms depending on $\phi_0$ that we will be able to combine with the left-hand side to construct a positive quadratic form. Secondly, we will work on the defaults that make the Hardy inequality
invalid in the 2D setting. It turns out that in this case again, we can play with radial coordinates and Fourier expansions in the angular variable. We will observe that the default of the Hardy inequality reduces to the 0 and first frequencies that we will have then to handle differently in comparison with larger frequencies. Moreover, we will observe that the nonlinearity is compatible with this splitting: the perturbation arising in the energy estimate for the $0$-mode does not depend on the $0$-mode. This will enable to buckle a contraction argument. 
In passing, to handle the nonlinearity, we will have to prove by a bootstrap argument that the non-zero modes of the perturbation ${\bf v}_a$ and ${\bf v}_b$ have to decay faster than expected. Details are provided in {\bf Section \ref{sec_uniqueness}}.

To make this method run, our assumptions contain two main restrictions. Firstly, we require the perturbations ${\bf v}_a$ and ${\bf v}_b$ to be small in appropriate weighted spaces. Asymptotically this has no impact (up to change $\beta_0$ for a smaller value $\beta'_0$). But we require here that in the bulk close to $B$ the perturbation is also small. Such a property is known for the solution ${\bf u}_a$ up to assume that the boundary data is a small perturbation of a reference flow. But it is still open for the arbitrary solution ${\bf v}_b.$  Even for small boundary data, standard a priori estimate techniques allow the existence of a ``large" perturbation. Secondly, {our uniqueness result is valid only for $\phi_0 \in (21/10,3].$} The assumption $\phi_0 \leq 3$ is fundamental: it allows the positivity of the quadratic form underlying our contraction argument. We expect that this restriction is related to a new bifurcation in the branch of solutions.
We could relax the assumption $\phi_0 > 21/10$  by imposing that the assumed faster decay is sufficiently strong (namely we could assume $\phi_0 > 2$ and $\beta_0$ sufficiently large). The restriction $\phi_0 > 21/10$ appears in the bootstrap argument. It enforces again that the Green functions -- associated with the linearized system around an asymptotic profile -- decay sufficiently fast to make the solutions ${\bf v}_b$ decay also sufficiently fast. We can then control trilinear terms arising in energy estimates by the positive-definite quadratic form at hand (a pseudo $\dot{H}^1$-norm).

The outline of the paper is as follows. In the next section, we prove our existence result. Since computations are very close to previous computations in \cite{HillairetWittwer} we will stick to the main new ingredients and recall the computations of  \cite{HillairetWittwer} for most technical parts. The core of the paper is the {\bf Section \ref{sec_uniqueness}} where we obtain our uniqueness result. Technical details (and especially the bootstrap argument involved in the proof of {\bf Theorem \ref{thm_uniqueness}}) are given in appendix.

\medskip

{\bf Acknowledgement.} The authors would like to thank Peter Wittwer for inspiring discussions at the origin of this study. The second author acknowledges support of Institut Universitaire de France.



\section{Existence theory} \label{sec_existence}
In this section, we provide a proof of {\bf Theorem \ref{thm_existence_intro}}. Our construction method follows closely the approach of \cite{HillairetWittwer} using the stream function/vorticity representation of the Navier-Stokes equations. We recall briefly the method for completeness and refer the reader to this previous contribution for further computations. 

Let ${\bf u}^* \in C^{\infty}(\partial B)$ and  $({\bf u},p)$ be an {\em a priori} associated smooth solution to \eqref{eq_NS2D}. We introduce:
\begin{equation} \label{eq_flux}
\phi_0 :=  \dfrac{1}{2\pi}\int_{\partial B} {\bf u}^* \cdot {\bf n} {\rm d}\sigma, \qquad \mu_0 := \dfrac{1}{2\pi} \int_{\partial B} {\bf u}^* \cdot {\bf n}^{\bot}{\rm d}\sigma,
\end{equation}
and we look for a solution close in the vicinity of $({\bf u}_{ref}[\phi_0,\mu],\, p_{ref}[\phi_0,\mu])$ where $\mu$ is close to $\mu_0.$ We split then ${\bf u} = {\bf u}_{ref}[\phi_0,\mu] + \tilde{{\bf u}},$ $p = p_{ref}[\phi_0,\mu] + \tilde{p}$ where $(\tilde{{\bf u}},\tilde{p})$ solves:
\begin{equation} \label{eq_NS2Dpertur}
\left\{
\begin{aligned}
& \left( {\bf \tilde{u}} \cdot \nabla\right) {\bf \tilde{u}} + {\bf u}_{ref}[\phi_0,\mu] \cdot \nabla {\bf \tilde{u}} + {\bf \tilde{u}} \cdot \nabla  {\bf {u}}_{ref}[\phi_0,\mu] + \nabla \tilde{p}  =  \Delta {\bf \tilde{u}}, && \text{in $\mathbb R^2 \setminus  \overline{B}$},\\
& {\rm div } \tilde{\bf u}  = 0, && \text{in $\mathbb R^2 \setminus  \overline{B}$}, \\
& \tilde{\bf u} = {\bf u}^*  - \left( \mu {\bf e}_{\theta} - \phi_0 {\bf e}_r \right), && \text{on $\partial B$}, \\
& \tilde{\bf u} = 0, && \text{at  infinity.}
\end{aligned}
\right.
\end{equation} 
Moreover, $\tilde{\bf u}$ is smooth and divergence-free on $\mathbb R^2 \setminus \overline{B}.$ Since it prescribes zero-flux on $\partial B$, it can be extended into a continuous piece-wise smooth divergence-free vector-field on $\mathbb R^2.$ We then find $\gamma \in  C^{\infty}(\mathbb R^2\setminus \overline{B})$ such that  $\tilde{\bf u} = \nabla^{\bot} \gamma.$  We take then the curl of the first equation in \eqref{eq_NS2Dpertur}, and consider the equations satisfied by the unknown
$\gamma$ together with $w = \nabla \times {\bf \tilde{u}} := \partial_1 \tilde{u}_2 - \partial_2 \tilde{u}_1.$ Going to radial coordinates, we have that $(\gamma,w)$ are $2\pi$-periodic solutions in $\theta$ for $r \in (1,\infty)$ to:
\begin{equation}
\left\{ 
\begin{aligned}
\partial _{rr}\gamma +\dfrac{1}{r}\partial _{r}\gamma +\dfrac{1}{r^{2}}%
\partial _{\theta \theta }\gamma & =-w,  \\
\partial _{rr}w+\dfrac{\left( \phi_0 +1\right) }{r}\partial _{r}w+\dfrac{1}{r^{2}}\partial _{\theta \theta }w-\dfrac{\mu }{r^{2}}\partial _{\theta }w & =\dfrac{\partial _{\theta }\gamma }{r}\partial _{r}w-\dfrac{\partial _{r}\gamma }{r} \partial _{\theta }w,
\end{aligned}
\right.
\label{gamma w}
\end{equation}
with the following periodic boundary conditions in $\theta$
\begin{equation}
\left\{ 
\begin{aligned}
\partial _{\theta }\gamma (1,\theta )=u_{r}^{\ast }(\theta )+\phi_0, &&  \text{ on $\mathbb T,$}  \\ 
\partial _{r}\gamma (1,\theta )=\mu -u_{\theta }^{\ast }(\theta ), && \text{ on $\mathbb T,$}  \\ 
\lim_{r\rightarrow \infty }\left( |\gamma (r,\theta )|+|\partial _{r}\gamma
(r,\theta )|\right) =0,&& \text{ on $\mathbb T,$} 
\end{aligned}
\right.
\label{bound gamma}
\end{equation}
where we denote $\mathbb T = \mathbb R /2\pi \mathbb Z$ the  1d-torus. 
Before constructing solutions to \eqref{gamma w}-\eqref{bound gamma}, we recall their relations with \eqref{eq_NS2D} in the following lemma:

\begin{lemma}
Assume $(\gamma,w) \in C^{\infty}((1,\infty) \times \mathbb T),$ satisfies \eqref{gamma w} and 
\begin{equation} \label{eq_conditiongamma}
\exists \, \alpha >0  \quad \text{ s.t. } \quad \sup_{(1,\infty) \times \mathbb R} \left(\sum_{k=0}^{3} r^{\alpha+k}|\nabla^k \gamma(r,\theta)|\right)  < \infty.
\end{equation}
There holds $\mathbf{u} = \nabla^{\bot} \gamma + {\bf u}_{ref}[\phi_0,\mu] \in C^{\infty}(\mathbb R^2 \setminus \overline{B})$
and there exists $p \in C^{\infty}(\mathbb R^2 \setminus \overline{B})$ for which 
\[
\left\{
\begin{aligned}
- \Delta {\bf u} + {\bf u} \cdot \nabla {\bf u} + \nabla p &= 0 && \text{in $\mathbb R^2 \setminus \overline{B},$}\\
{\rm div} {\bf u} &= 0 && \text{in $\mathbb R^2 \setminus \overline{B}.$}
\end{aligned}
\right.
\]
\end{lemma}

\begin{proof}
We follow closely the proof of \cite[Theorem 2]{HillairetWittwer}. The smoothness of ${\bf u}$ is straightforward. Then, by construction, equations \eqref{gamma w} entail that :
\[
\nabla \times \left( - \Delta {\bf u} + {\bf u} \cdot \nabla {\bf u} \right) = 0, \quad
\text{ in $\mathbb R^2 \setminus \overline{B}.$}
\]
Similarly to the construction of $\gamma$ from ${\bf u}$ (up to a rotation of angle $\pi/2$), we infer that there exists $p \in C^{\infty}(\mathbb R^2 \setminus \overline{B})$ and $\lambda \in \mathbb R$ for which:
\[
- \Delta {\bf u} + {\bf u} \cdot \nabla {\bf u} + \nabla p = \lambda \dfrac{x^{\bot}}{|x|^2}.
\]
We remark then that:
\[
\lambda = \frac{r}{2\pi}\int_{\partial B(0,r)} (- \Delta {\bf u} + {\bf u} \cdot \nabla {\bf u}) \cdot {\bf e}_{\theta} {\rm d}\sigma, \qquad \forall \, r > 1.
\]
However, by construction $\Delta {\bf u}$ and ${\bf u} \cdot \nabla {\bf u}$ decay faster than $1/r^{3}$ at infinity so that $\lambda =0.$ This ends the proof.
\end{proof}

\medskip

In what follows, all the constructed solutions $(\gamma,w)$ to \eqref{gamma w} will match the condition \eqref{eq_conditiongamma} so that the above result applies.

\subsection{Rewriting of the main equations} \label{sec_rewriting}
We now look for solutions $(\gamma,w)$ to \eqref{gamma w} in terms of Fourier series of the angular coordinate $\theta:$ 
\begin{equation*}
\gamma (r,\theta )=\sum\limits_{n\in \mathbb{Z}}\gamma _{n}(r)e^{in\theta },%
\text{ \ \ \ }w(r,\theta )=\sum_{n\in \mathbb{Z}}w_{n}(r)e^{in\theta },
\end{equation*}%
and similarly, we compute Fourier modes of the boundary conditions:
\[
v_{r,n}^{\ast } = \frac{1}{2\pi }\int_{0}^{2\pi }\left( u_{r}^{\ast
}(\theta )+\phi_0 \right) e^{-in\theta }{\rm d}\theta ,  \quad
v_{\theta ,n}^{\ast } = \frac{1}{2\pi }\int_{0}^{2\pi }\left( u_{\theta
}^{\ast }(\theta )-\mu \right) e^{-in\theta }{\rm d}\theta \,, \quad \forall \, n \in \mathbb Z.
\]
We note that, because of the convention \eqref{eq_flux} we have $v^*_{r,0} =0$ and ${v}^{*}_{\theta,0} = \mu_0 - \mu.$ 

\medskip

Plugging into \eqref{gamma w} and identifying Fourier modes transforms \eqref{eq_NS2D} into a discrete family of coupled differential systems in $(1,\infty).$
For all $n\in \mathbb Z$ we obtain:
\begin{equation} 
\left\{ 
\begin{aligned}
\partial _{rr}\gamma _{n}+\dfrac{1}{r}\partial _{r}\gamma _{n}-\dfrac{n^{2}}{%
r^{2}}\gamma _{n}=-w_{n}, && \text{ in $(1,\infty)$} \\ 
\partial _{rr}w_{n}+\frac{\left( \phi_0 +1\right) }{r}\partial _{r}w_{n}-\frac{%
in\mu +n^{2}}{r^{2}}w_{n}=F_{n}, && \text{in $(1,\infty)$}%
\end{aligned}%
\right. \label{odes}
\end{equation}%
with boundary conditions: 
\begin{equation}
\left\{ 
\begin{aligned}
in\gamma _{n}(1)& =v_{r,n}^{\ast }, \\ 
-\partial _{r}\gamma _{n}(1)& =v_{\theta ,n}^{\ast }, \\ 
\lim_{r\rightarrow \infty }\left( |\gamma _{n}(r)|+|\partial _{r}\gamma
_{n}(r)|\right) & =0,%
\end{aligned}%
\right.
\label{bound gamma n}
\end{equation}%
and where the coupling term $F_{n}$ reads:
\begin{equation}
F_{n}(r)=\frac{i}{r}\sum\limits_{k+l=n}\left( l\gamma _{l}\partial
_{r}w_{k}-\partial _{r}\gamma _{l}kw_{k}\right).  \label{fn}
\end{equation}

Given suitable $(v^*_{r,n},v^*_{\theta,n})_{n\in \mathbb Z}$ (to be made precise below), our formal reasoning to solve \eqref{odes}-\eqref{bound gamma n}-\eqref{fn} reads as follows. We will look for solutions to \eqref{odes}-\eqref{bound gamma n} such that
$F_n$ is computed from the solution itself via \eqref{fn}. This is typically a fixed-point problem. We must expect that the main difficulty to prove existence of a solution is to find a framework that enables to control the decay of solutions at infinity. To highlight such a construction,  let us forget that we have sums over the index $n$ and assume that we have polynomially decaying functions. If all $\gamma_n$ decay like $1/r^{\alpha}$ then one must expect that the $F_n$ decay like $1/r^{2\alpha+4}.$ Solving the differential system   \eqref{odes}-\eqref{bound gamma n} yields a solution $(\tilde{\gamma}_n,\tilde{w}_n)$ 
such that:
\begin{itemize}
\item $\tilde{w}_n$ splits into a Green function that decays like $r^{\zeta^-_n}$ and a solution resulting from the source term $F_n.$  One reads directly from the equation that the solution resulting from the source term decays two power less than the source term meaning $1/r^{2\alpha+2}$ and:
\begin{equation} \label{eq_zetan}
\zeta_n^{-} = - \left( \frac{\phi_0 }{2} +  \frac{1}{2}\sqrt{\phi_0 ^{2}+4(in\mu
+n^{2})} \right)
\end{equation}
\item similarly, $\tilde{\gamma}_n$ splits into a Green function that decays like $r^{-|n|}$ and a solution resulting from the source term $\tilde{w}_n.$  One reads directly from the equation that the solution resulting from the source term decays two power less than the source term meaning $r^{\max(\zeta_n^{-} +2,-2\alpha)}.$
\end{itemize}
Anticipating a recursive process, we may expect convergence of the iteration $(\gamma_n,w_n) \to (\tilde{\gamma}_n,\tilde{w}_n)$ only if the new decay (of $\tilde{\gamma}_n$) is faster than the previous one and consistent with the assumption $\alpha >0$. For this, we require that $\max (-|n|, \max (\Re(\zeta_n^-)+2, -2\alpha)) <- \alpha.$ for all $n\in \mathbb Z.$ This can be recast into the three conditions
\[ 
|n| > \alpha \text{ and } \Re(\zeta_n^-) < - (2+\alpha) \text{ and } 2 \alpha > \alpha \,, 
\] 
that must hold for all $n \in \mathbb Z.$ We note here that the exponent $\alpha >0$ 
is free and to be chosen so that these conditions are satisfied. In this respect, the last condition $2\alpha > \alpha$ is free since we must choose $\alpha > 0.$
The first  condition can never be reached for all $n\in\mathbb Z$ because of the case $n=0$. To overcome this difficulty, we proceed like in the previous reference \cite{HillairetWittwer}: we first assume that we can construct a mode $(\gamma_0,w_0)$ decaying as fast as suitable. We need then to match the above conditions only when $n \neq 0.$ Because of the explicit formula for $\zeta_n^{-},$ we remark we can choose an $\alpha > 0$ so that it holds true in case {$\Re(\zeta_1^{-}) < - 2$} that writes:
\begin{equation} \label{eq_conditionphimu}
\phi_0 > \frac 32 \qquad \text{ or } \qquad \left( \phi_0 \in \left[0, 3/2 \right] \text{ and } |\mu| > (4-\phi_0) \sqrt{3-2\phi_0} \right).  
\end{equation}

We then consider the $0$-mode. We remark that this frequency behaves differently whether $\phi_0$ is larger or smaller than $2.$  When $\phi_0 \leq 2$, we have $\Re(\zeta_0^-) \geq -2 > -2-\alpha$ (whatever $\alpha>0$). Hence, the only possible formula for a solution to \eqref{odes} that decays sufficiently fast at infinity is the one given by the nonlinearity. This one decays with a power $2\alpha > \alpha$. By choosing this only solution, we need to relax the boundary condition on $\partial_r \gamma_0(1)$ that we recover in a last time playing on the parameter $\mu$ like in \cite{HillairetWittwer}.  But one novelty arises when $\phi_0 > 2.$ Indeed, we have then $\Re(\zeta_0^{-}) < -2$ so that we can allow a Green-function of the $w_0$-equation in the solution (up to restrict the size of $\alpha$).  We are then able to fix a solution that matches the value of $\partial_r \gamma_0(1).$ So, we can solve the full problem with boundary data for any value of $\mu.$ Consequently, playing with the parameter $\mu$
yields various solutions. The fact that these solutions are different entails from their asymptotic first order when $r \to \infty.$
 
\subsection{Function spaces} 
To enter into the details of the construction, we introduce now spaces inspired of \cite{HillairetWittwer} in which we will perform our fixed-point argument.
\begin{definition}
\label{defi space}Given $\kappa >0,$ $\alpha >0$ and $m\in \mathbb{N},$ such
that $m<\kappa ,$ we set:

\begin{align*}
\mathcal{B}_{\kappa }& :=\{\hat{\varphi}^{\ast }\in \mathbb{C}^{\mathbb{Z}}%
\text{ such that }\sup_{n\in \mathbb{Z}}(1+|n|)^{\kappa }|\varphi _{n}^{\ast
}|<\infty \}, \\
\mathcal{B}_{\kappa }^{0}& :=\{\hat{\varphi}^{\ast }\in \mathcal{B}_{\kappa }%
\text{ such that }\varphi _{0}^{\ast }=0\},
\end{align*}%
and%
\begin{align*}
\mathcal{B}_{\alpha ,\kappa }& :=\{\hat{\varphi}\in (C[1,\infty );\mathbb{C}%
)^{\mathbb{Z}},\text{ such that }\sup_{n\in \mathbb{Z}}\sup_{r\in \lbrack
1,\infty )}r^{\alpha }(1+|n|)^{\kappa }|\varphi _{n}(r)|<\infty \}, \\
\mathcal{U}_{\alpha ,\kappa }^{m}& :=\{\hat{\varphi}\in (C^{m}[1,\infty );%
\mathbb{C})^{\mathbb{Z}},\text{ such that }\left( \partial _{r}^{l}\varphi
_{n}\right) _{n\in \mathbb{Z}}\in \mathcal{B}_{\alpha +l,\kappa -l},\text{
for all }0\leq l\leq m\}.
\end{align*}%
\end{definition}

These function spaces are reminiscent of weighted Sobolev spaces, and permit
to obtain sharp estimates on the decay of solutions to (\ref{odes})-(\ref{bound gamma n}). The spaces with one lower index are used for boundary data, whereas the spaces
with two lower indices (mainly $\mathcal{U}_{\alpha ,\kappa }^{m}$) will be
used for solving (\ref{odes})-(\ref{bound gamma n}). We recall that these spaces are decreasing for inclusion with continuous embeddings in the  parameters $\kappa,\alpha,m$ independently. We introduce a particular naming of solutions to \eqref{eq_NS2D} obtained through this process:

\begin{definition}
Let $\kappa \in  (0,\infty)$ and a boundary
condition $\hat{\mathbf{v}}^{*}\colon =(\hat{v}_{r}^{*},\hat{v}%
_{\theta }^{* })\in \mathcal{B}_{\kappa+5}^{0}\times \mathcal{B}%
_{\kappa+4}^{0}.$ 
Given $(\phi_0,\mu) \in [0,\infty) \times \mathbb R,$ we call $\kappa$-solution for the boundary conditions $\hat{\mathbf{v}}^{*},$ the flux $\phi_0$ and the angular velocity $\mu $ a
pair $(\hat{\gamma},\hat{w}),$ such that 

\begin{itemize}
\item $(\hat{\gamma},\hat{w}) \in \mathcal{U}_{\alpha_0,\kappa +4}^{2}\times \mathcal{U}%
_{\alpha_0+2,\kappa +3}^{2}$ for some $\alpha_0 >0,$

\item all modes of $(\hat{\gamma},\hat{w})$ solve \eqref{gamma w} with boundary condition \eqref{bound gamma n} and source term given by \eqref{fn}.
\end{itemize}
\end{definition}

In this setting, our main result reads:

\begin{theorem} \label{thm_existencekappa}
Let $\phi_0 \in [0,\infty)$ and $\mu_0 \in \mathbb R$ satisfying \eqref{eq_conditionphimu}. Given $\kappa >0,$ 
there exists $\epsilon>0$ depending on $\phi_0, \mu_0, \kappa$ and an open interval $I_{\kappa ,\mu_0}$ which
contains $\mu_{0}$ such that for arbitrary $\hat{\mathbf{v}}^{\ast }\in \mathcal{B}_{\kappa +5}^{0}\times \mathcal{B}_{\kappa +4}^{0}$
with norm smaller than $\epsilon$ the following holds true 
\begin{itemize}
\item[if] $0\leq \phi_0 \leq 2,$ there exists a circulation $\mu$ close to $\mu_0$ such that there is a $\kappa$-solution $(\hat{\gamma},\hat{w})$ for the boundary condition $\hat{\mathbf{v}}^{*}$ 
\item[if] $\phi_0 > 2,$ whatever the circulation $\mu$ close to $\mu_0,$ there exists a $\kappa$-solution $(\hat{\gamma},\hat{w})$ for the boundary condition $\hat{\mathbf{v}}^{*}.$
\end{itemize}
\end{theorem}
We obtain {\bf Theorem \ref{thm_existencekappa}} by splitting \eqref{odes} into 
\begin{itemize}
\item computing the source term $(F_n)_{n\in \mathbb Z}$ via \eqref{fn}
\item solving the family of differential systems \eqref{odes} with boundary conditions \eqref{bound gamma n} and arbitrary data $(F_n)_{n\in\mathbb N}$
\end{itemize}
We shall then conclude with a fixed point argument in terms of $(\hat{\gamma},\hat{w})$ in $\mathcal{U}_{\alpha_0,\kappa +4}^{2}\times \mathcal{U}_{\alpha_0 +2,\kappa+3}^{2}$ for a well-chosen $\alpha_0.$ 

\medskip

To start with, we recall that computing the source terms $(F_n)_{n\in\mathbb N}$ has been established in \cite[Lemma 5]{HillairetWittwer}:

\begin{lemma} \label{lem_nonlinearity}
\label{map2}Let $\alpha>0$ and $\kappa>0$. Then, the mapping $\mathcal{NL}:(%
\mathcal{U}_{\alpha ,\kappa +4}^{2}\times \mathcal{U}_{\alpha +2,\kappa
+ 2}^{2})^{2}\rightarrow \mathcal{B}_{4+2\alpha ,\kappa +1},$ defined by%
\begin{equation*}
\mathcal{NL}[(\hat{\gamma}^{a},\hat{w}^{a}),(\hat{\gamma}^{b},\hat{w}%
^{b})]=\left( r\mapsto -\frac{i}{r}\sum\limits_{k+l=n}(kw_{k}^{a}(r)\partial
_{r}r_{l}^{b}(r))(lr_{l}^{a}(r)\partial _{r}w_{k}^{b}(r))\right) _{n\in 
\mathbb{Z}},
\end{equation*}%
is bilinear and continuous.
\end{lemma}

To complete the proof of {\bf Theorem \ref{thm_existencekappa}}, we analyze at first the linearized system in the following subsection and apply the subsequent continuity results in the following part. 

\subsection{Analysis of the linearized problem}

In this section, we fix a set of boundary data $(\hat{v}_r^*,\hat{v}_{\theta}^*) \in \mathcal B_{\kappa+5}^0 \times \mathcal B_{\kappa+4}^0$ and of source term 
$\hat{F} \in  { \mathcal B_{4+2\alpha,\kappa+1}}$ -- with $\kappa$ and $\alpha$ to be made precise -- and we consider the infinite differential system \eqref{odes} on $(1,\infty)$ with the boundary conditions:%
\begin{equation} \label{eq_boundarygamman}
\left\{ 
\begin{array}{l}
in\gamma _{n}(1)=v_{r,n}^{\ast }, \\ 
-\partial _{r}\gamma _{n}(1)=v_{\theta ,n}^{\ast }, \\ 
\lim_{r\rightarrow \infty }\left( |\gamma _{n}(r)|+|\partial _{r}\gamma
_{n}(r)|\right) =0,%
\end{array}%
\right.  \qquad n \neq 0,
\end{equation}%
and
\begin{equation} \label{eq_boundarygamman0}
\left\{ 
\begin{array}{l}
-\partial _{r}\gamma _{0}(1)= \mu_0 - \mu, \\ 
\lim_{r\rightarrow \infty }\left( |\gamma _{0}(r)|+|\partial _{r}\gamma
_{0}(r)|\right) =0,%
\end{array}%
\right.  \qquad n=0.
\end{equation}%

We explain at first how we solve this differential system and we state then our main result.
\medskip

\subsubsection{The case $n\neq 0$} 
To solve the $w_n$ equation, we remark that the associated Green functions read 
$r^{\zeta_n^{\pm}}$ with:
\begin{equation*}
\zeta _{n}^{\pm }=-\frac{\phi_0 }{2}\pm \frac{1}{2}\sqrt{\phi_0 ^{2}+4(in\mu
+n^{2})}
\end{equation*}%
while the Green functions associated with the $\gamma_n$ equation are $r^{\pm |n|}.$
Taking into account that the solutions must decay at infinity we obtain the following explicit expressions for $r \in (1,\infty).$ First, we have:
\begin{equation} \label{eq_wn}
w_n(r) = \bar{w}_n r^{\zeta_n^-} - w_n[F_n](r)
\end{equation}
where:
\begin{equation} \label{eq_wnF}
w_n[F_n](r) =  \int_{r}^{\infty }\frac{sF_{n}(s)}{%
\left( \sqrt{\phi_0 ^{2}+4(in\mu +n^{2})}\right) }\left( \frac{r}{s}\right)
^{\zeta _{n}^{+}}ds \\ + \int_{1}^{r}\frac{sF_{n}(s)}{\left( \sqrt{\phi_0
^{2}+4(in\mu +n^{2})}\right) }\left( \frac{r}{s}\right) ^{\zeta _{n}^{-}}ds.
\end{equation}
Then, we deduce that 
\begin{equation} \label{eq_gamman}
\gamma_n(r) = 
\left\{
\begin{aligned}
& \frac{\bar{\gamma}_{n}}{r^{|n|}} 
- \frac{\bar{w}_n}{(\zeta_n^{-} +2)^2 - n^2} r^{2+\zeta_n^-}
- \gamma_n[F_n](r),  && \text{if } (\zeta_n^{-}+ 2)^2 \neq n^2\\
&  \frac{\bar{\gamma}_{n}}{r^{|n|}} 
+  \frac{\bar{w}_n}{{2}|n| r^{|n|}}\ln r - \gamma_n[F_n](r), 
&& \text{if } \zeta_n^{-}+ 2 + |n| = 0
\end{aligned}
\right.
\end{equation}
where
\begin{equation} \label{eq_gammanF}
\gamma_n[F_n](r) = \int_{r}^{\infty }\frac{sw_{n}[F_n](s)}{2|n|}\left( \frac{r}{s}\right) ^{|n|}ds+\int_{1}^{r}\frac{sw_{n}[F_n](s)}{2|n|}\left( \frac{s}{r}\right)^{|n|}ds.
\end{equation}
We point out that $2+ \Re(\zeta_n^-)$ will be negative in our approach so that we do not have to consider the case $\zeta_n^-+2 =|n|.$\footnote{We complement this remark with a word concerning the condition $\zeta_n^- + 2 =  - |n|.$ Standard computations show that this resonance happens only when
$\mu= 0,$ and one value of $|n|$ depending on $\phi_0$ ranging a countable subset of $(2,4).$ 
}
Finally, we fix $(\bar{\gamma}_n,\bar{w}_n) \in \mathbb R^2$ in order to match the remaining boundary conditions in \eqref{eq_boundarygamman}. Standard computations show that:

	\begin{itemize}
		\item if $\zeta_n^- + 2 \neq -n$, we have:
		\begin{equation} \label{eq_constantgammaw}
			\left\{ 
			\begin{aligned}
				\bar{\gamma}_{n} & =  -\dfrac{1}{2+\zeta_n^- + |n|} \left( (2+\zeta_n^-) (i {\rm {sgn}}(n)\frac{v_{r,n}^*}{|n|}-\gamma_n[F_n](1)) +  \partial_r \gamma_n[F_n](1)-v_{\theta,n}^* \right), \\
				\bar{w}_n &= -\dfrac{(\zeta_n^- +2)^2 - n^2}{2+\zeta_n^- + |n|} \left( |n|\gamma_n[F_n](1)-i {\rm {sgn}}(n) v_{r,n}^* +\partial_r \gamma_n[F_n](1) -v_{\theta,n}^*  \right).
			\end{aligned}
			\right.
		\end{equation}%
		\item if $\zeta_n^- + 2 = -|n|$
		\begin{equation} \label{eq_constantgammawresonant-new}
			\left\{ 
			\begin{aligned}
				\bar{\gamma}_{n} & =-i {\rm {sgn}}(n)  \frac{v_{r,n}^*}{|n|} + \gamma_n[F_n](1) , \\
				\bar{w}_n &= 2|n| \left(  |n|\gamma_n[F_n](1)- \partial_r \gamma_n[F_n](1)-v_{\theta,n}^* - i {\rm {sgn}}(n)v_{r,n}^*  \right).
			\end{aligned}
			\right.
		\end{equation}%

	\end{itemize}

\subsubsection{The case $n=0$} {In case $n=0$ we have the differential system:}
\begin{equation}
\left\{ 
\begin{aligned}
\partial _{rr}\gamma _{0}+\frac{1}{r}\partial _{r}\gamma _{0}=-w_{0}, \\ 
\partial _{rr}w_{0}+\frac{\left( \phi_0 +1\right) }{r}\partial _{r}w_{0}=F_{0},%
\end{aligned}%
\right. \text{ \ in }(1,\infty ),  \label{zero mode}
\end{equation}%
with a decay condition at infinity and only one boundary condition:
\begin{equation} \label{eq_boundgamma_0}
-\partial_r \gamma_0(1) =  \mu_0 - \mu.
\end{equation}
The non-constant Green function of the $w_0$ equation  reads $r^{-\phi_0}.$ Consequently, we distinguish between two subcases.

\medskip

{\bf Subcase 1 : $0\leq \phi_0 \leq 2.$} The Green function does not decay sufficiently fast to be integrated in the $\gamma_0$ formula. In this case, our only choice for $w_0$ is the solution resulting from the source term that reads:
\begin{align} \label{eq_w0<2}
w_0(r) &=  w_0[F_0](r) := \int_{r}^{\infty }\left(
\int_{s}^{\infty }\left( \frac{t}{s}\right) ^{\phi_0 +1}F_{0}(t)dt\right) {\rm d}s,\\
\label{eq_gamma0<2}
\gamma_0(r) &= \gamma_0[F_0](r) := \int_{r}^{\infty} \int_{s}^{\infty} \dfrac{\sigma}{s} w_0[F_0](\sigma){\rm d}\sigma {\rm d}s.
\end{align}
This does not enable to choose the value for $\partial_r \gamma_0(1)$ that reads:
\[
\partial_r \gamma_0(1) = \int_{1}^{\infty} s w_0[F_0](s){\rm d}s.
\]

{\bf Subcase 2: $\phi_0 > 2.$} We can fix in this case 
\begin{align} \label{eq_w0>2}
w_0(r) = \dfrac{\bar{w}_0}{r^{\phi_0}} + w_0[F_0](r),
\\
\label{eq_gamma0>2}
\gamma_0(r) =  \dfrac{\bar{w}_0}{(\phi_0-2)^2 r^{\phi_0-2}} + \gamma_0[F_0](r),
\end{align}
where the constant $\bar{w}_0$ is chosen to match the boundary condition \eqref{eq_boundgamma_0} leading to:
\begin{equation} \label{eq_constantgamma0}
\bar{w}_0 = (\phi_0-2) \left( (\mu_0 - \mu) - \int_{1}^{\infty} s w_0[F_0](s){\rm d}s\right). 
\end{equation}

\subsubsection{Result on the analysis of the linearized system}
{
Given $(\phi_0,\mu) \in [0,\infty) \times \mathbb R,$ we define the mapping $\mathcal S_{\mu,\phi_0} : (\hat{F},\hat{v}_{r}^*,\hat{v}_{\theta}^{*}) \to (\hat{\gamma},\hat{w})$ where $\hat{w} = (w_n)_{n\in\mathbb Z}$ and $\hat{\gamma} = (\gamma_n)_{n\in\mathbb Z}$ are computed through formulas \eqref{eq_wn}-\eqref{eq_wnF} and \eqref{eq_gamman}-\eqref{eq_gammanF} with constants given by \eqref{eq_constantgammaw} and \eqref{eq_constantgammawresonant-new} in case $n \neq 0,$ 
while the $0$-mode is given by formula \eqref{eq_w0<2}-\eqref{eq_gamma0<2} in case $0\leq \phi_0 \leq 2,$
and \eqref{eq_w0>2}-\eqref{eq_gamma0>2}-\eqref{eq_constantgamma0} in case $\phi_0 > 2.$ This mapping associates to the data $ (\hat{F},\hat{v}_{r}^*,\hat{v}_{\theta}^{*})$ solutions $ (\hat{\gamma},\hat{w})$ to \eqref{gamma w}. Furthermore, these solutions match the boundary condition \eqref{bound gamma n} for $n \neq 0$ and  also for $n=0$ in case $\phi_0 >2.$ 
}
\medskip

The analysis of this mapping is the content of the next lemma:

\begin{lemma} \label{lem_analysislinearproblem}
Let $\phi_0 \in [0,\infty)$ 
\begin{itemize}
\item[i)] Let $\mu \in \mathbb R$ such that \eqref{eq_conditionphimu} holds true and 
$\alpha \in (0,-\Re(\zeta_1^-)-2).$
The mapping $\mathcal S_{\mu,\phi_0}: \mathcal{B}_{4+2\alpha ,\kappa +1}\times 
\mathcal{B}_{\kappa +5}^{0}\times \mathcal{B}^0_{\kappa +4} \rightarrow 
\mathcal{U}_{\alpha ,\kappa +4}^{2}\times \mathcal{U}_{\alpha +2,\kappa
	+3}^{2},$ which associates to the triple $(\hat{F},\hat{v}_{r}^*,\hat{v}_{\theta}^{*})$ the pair $ (\hat{\gamma},\hat{w})$ is linear and continuous.
\item[ii)] Let $I \subset [0,\infty)$ be an interval  and assume $\alpha \in (0,
\inf_{\mu \in I} (-\Re(\zeta_1^{-}))-2).$  Then, the mapping $\mu \mapsto \mathcal S_{\mu,\phi_0}$ is continuous from $I$ with values into the set of linear continuous mapping $\mathcal{B}_{4+2\alpha ,\kappa+1}\times 
\mathcal{B}_{\kappa +5}^{0}\times \mathcal{B}^0_{\kappa +4} \rightarrow 
\mathcal{U}_{\alpha ,\kappa +4}^{2}\times \mathcal{U}_{\alpha +2,\kappa
+3}^{2}.$
\end{itemize}

\end{lemma}

We emphasize that for $(\phi_0,\mu) \in [0,\infty) \times \mathbb R,$ we have
\[
-\Re(\zeta_1^-) = \frac{\phi_0 }{2}+\frac{1}{2\sqrt{2}}\left( \phi_0 ^{2}+4+%
\sqrt{\left( \phi_0 ^{2}+4\right) ^{2}+16\mu^{2}}\right) ^{1/2} .
\]
Consequently,  \eqref{eq_conditionphimu} entails that the family of exponents $\alpha$ for which the assumption of $i)$ holds true is not empty.
The proof of {\bf Lemma \ref{lem_analysislinearproblem}} follows the same strategy as in \cite[Section 5.1]{HillairetWittwer}. We point out the main differences in {\bf Appendix \ref{app_linearized}} and proceed to the proof of our main result.

\subsection{Proof of Theorem \ref{thm_existencekappa}}

Let $\phi_0 \in [0,\infty)$ and $\mu_0 \in \mathbb R$ satisfying \eqref{eq_conditionphimu}. By continuity, there exists $\varepsilon_0$ such that
\[
\frac{\phi_0 }{2}+\frac{1}{2\sqrt{2}}\left( \phi_0 ^{2}+4+%
\sqrt{\left( \phi_0 ^{2}+4\right) ^{2}+16\mu^{2}}\right) ^{1/2} - 2  > 0,  \qquad 
\forall \, \mu \in [\mu_0-\varepsilon_0,\mu_0+ \varepsilon_0].
\] 
In particular, we may fix:
\[
\alpha_0 = \dfrac{1}{2} \min_{\mu \in  [\mu_0-\varepsilon_0,\mu_0+ \varepsilon_0]} \left\{ \frac{\phi_0 }{2}+\frac{1}{2\sqrt{2}}\left( \phi_0 ^{2}+4+%
\sqrt{\left( \phi_0 ^{2}+4\right) ^{2}+16\mu^{2}}\right) ^{1/2} - 2 \right\} \in (0,\infty).
\]

Consider now $\kappa >0$ and $(\hat{v}_r^*,\hat{v}_{\theta}^*) \in \mathcal B^0_{\kappa+5} \times \mathcal B^0_{\kappa+4}$.
Thanks to {\bf Lemma \ref{lem_nonlinearity}} and {\bf Lemma \ref{lem_analysislinearproblem} } we can construct the mapping 
\[
\begin{array}{rrcl}
\Phi_{\mu} :&  \mathcal{U}_{\alpha_0 ,\kappa +4}^{2}\times \mathcal{U}_{\alpha_0 +2,\kappa
+3}^{2} & \longrightarrow &  \mathcal{U}_{\alpha_0,\kappa +4}^{2}\times \mathcal{U}_{\alpha_0 +2,\kappa
+3}^{2} \\[6pt]
&  (\hat{\gamma},\hat{w}) &\longmapsto &  \mathcal S_{\mu,\phi_0} (\mathcal{NL}(\hat{\gamma},\hat{w}) , \hat{\bf{v}}^*).
\end{array}
\] 
Our existence  result is based on the remark that a fixed point of $\Phi_{\mu}$ yields a solution to \eqref{gamma w}-\eqref{bound gamma n}-\eqref{fn} in case $\phi_0 > 2$ while in case $0\leq \phi_0 \leq 2$ we must 
complement the construction of a fixed point to $\Phi_{\mu}$ by matching the boundary conditions for the zero mode.

In case $\phi_0 > 2,$ since $\mathcal{NL}$ and $\mathcal S_{\mu,\phi_0}$ are (bilinear and linear respectively) continuous, uniformly in $\mu \in [\mu_0-\varepsilon_0,\mu+\varepsilon_0],$ existence of a fixed point for small data $\hat{\bf v}^* \in \mathcal B_{\kappa+5}^0 \times \mathcal B_{\kappa+4}^{0}$ and small $|\mu-\mu_0|$ yields from a classical Picard fixed-point theorem (up to take $\varepsilon_0$ smaller if necessary).

In case $0\leq \phi_0 \leq 2,$ we remark that a classical (parameter)-dependent fixed-point argument yields that, for small $\varepsilon >0$ we have existence of a unique fixed-point $(\hat{\gamma}_{\mu},\hat{w}_{\mu})$ to $\Phi_{\mu}$ whatever $\hat{\bf v}^*$ with norm smaller than $\varepsilon$ for all $\mu \in  [\mu_0-\varepsilon_0,\mu_0+\varepsilon_0].$ Furthermore:
\begin{itemize}
\item There exists a constant $C_0$ independent of $\mu,{\hat{\bf v}}^*$ such that 
\[
\|(\hat{\gamma}_{\mu},\hat{w}_{\mu})\|_{\mathcal{U}_{\alpha ,\kappa +4}^{2}\times \mathcal{U}_{\alpha +2,\kappa
+3}^{2}} \leq C_0 (\|\hat{v}_r^*; {\mathcal B}_{\kappa+5}^0 \| + \|\hat{v}_{\theta}^*; {\mathcal B}_{\kappa+4}^0 \|).
\] 
\item The mapping $\mu \mapsto (\hat{\gamma}_{\mu},\hat{w}_{\mu})$ is continuous.  
\end{itemize}
This entails that the mapping $\mu \mapsto \mu - \partial_r \gamma_{\mu,0}(1)$ is a continuous perturbation of the identity from $[\mu_0-\varepsilon_0,\mu_0+\varepsilon_0]$ into $\mathbb R$ where:
\[
|\partial_r \gamma_{\mu,0}(1)| \leq C_0 (\|\hat{v}_r^*; {\mathcal B}_{\kappa+5}^0 \| + \|\hat{v}_{\theta}^*; {\mathcal B}_{\kappa+4}^0 \|)^2
\]
for some constant $C_0$ (possibly different to the previous one but again) independent of $\mu$ and $\hat{\bf v}^*.$
Up to choose $\varepsilon$ sufficiently small with respect to $\varepsilon_0$ we have in particular that $\mu_0$ is in the range of this mapping. This concludes the proof.



\section{Uniqueness of steady solutions}
\label{sec_uniqueness}

We provide here a proof of {\bf Theorem \ref{thm_uniqueness}.}
We recall that we assume $\phi_0 \in (21/10,3)$ and $\beta_0 > 0.$
We fix $\mu \in \mathbb R$ and ${\bf u}^* \in C^{\infty}(\partial B)$ a boundary condition.  We assume $({\bf u}_{a},p_a)$ and $({\bf u}_b,p_b)$ are two smooth solutions to \eqref{eq_NS2D} with the same boundary data ${\bf u}^*$ and such that:
\begin{equation}\label{eq_expansionu}
{\bf v}_{i} = {\bf u}_{i} - {\bf u}_{ref}[\phi_0,\mu],\quad i=a, \, b,
\end{equation}
satisfies:
\begin{equation} \label{eq_Miref}
M_{i}^{ref} := \sup_{|x| > 1} |x|^{1+\beta_0} \left(|{\bf v}_i(x)| + |x| |\nabla {\bf v}_i(x)| + 
|x|^{2}|\nabla^2 {\bf v}_{i}(x)|\right) < \infty.
\end{equation}
Our aim is to prove that, if $M_{a}^{ref},\, M_{b}^{ref}$ and $\mu$ are sufficiently small, then ${\bf u}_a = {\bf u}_b.$
We want to apply an energy method but a standard multiplier argument is not sufficient (see the introduction). To get a finer description we will use a similar expansion to the previous section by introducing radial coordinates and using Fourier modes in the 
angular variable. We will then treat differently small frequencies and large frequencies in terms of the angular components of the velocity field. 

\medskip

Precisely, given a generic divergence-free ${\bf w} \in C^{\infty}(\mathbb R^2 \setminus \overline{B})$ that prescribes no flux through circles, we previously used that 
\[
{\bf w} = \nabla^{\bot} \varphi \quad {\text{with} } \quad
\varphi(r,\theta)= \sum_{k \in \mathbb Z} \varphi_{k}(r) \exp(ik\theta). 
\]
Correspondingly, we expand:
\begin{equation*}
	\mathbf{w}  = \mathbf{w}
	_{0} + \mathbf{w}^{(0)}= \mathbf{w}_{0} + \mathbf{w}_{1} + \mathbf{w}^{(1)},
\end{equation*}
where
\begin{equation*}
	\begin{aligned}  {\bf w}_0(r,\theta) & = \nabla^{\bot} \varphi_0(r)=  w_0(r) 
		{\bf e}_{\theta} , \\[6pt]  {\bf w}_1(r,\theta) &= \nabla^{\bot} [\varphi_1(r) \exp{i\theta} + \varphi_{-1}(r)\exp(-i\theta)]\\
		&= \left( w_c^{(n)}(r) \cos\theta + w_s^{(n)} \sin\theta\right) {\bf e}_{r} \\
		& \quad + \left( w_c^{(t)}(r) \cos\theta + w_s^{(t)} \sin\theta\right) {\bf e}_{\theta} ,
	\end{aligned}
\end{equation*}
and, when $k \geq 0$ :
\begin{equation*}
	\mathbf{w}^{(k)} = \sum_{|\ell| > k} \nabla^{\bot} [ \varphi_{\ell}(r) \exp(i
	\ell \theta)].
\end{equation*}
Such an expansion is in particular valid for our candidate perturbations ${\bf v}_{a}$ and ${\bf v}_b$.  To prepare a fine treatment of nonlinear terms, we start by improving the known decay at infinity of these perturbations. This is the content of the following proposition:
\begin{proposition} \label{prop_decaymode1}
Given $\beta_1 \in (1/2,1),$ and $\varepsilon \in (0,1)$ there exist positive constants $C_0,\, C_1$ and a positive exponent $k$  such that:
\begin{align*}
\sup_{|x| > 1} [ |x|^{1+\beta_1} (|{\bf v}_{b,1}(x)|+ |x||\nabla {\bf v}_{b,1}(x)|) & + |x|^{1+2\beta_1-\varepsilon} (|{\bf v}_b^{(1)}(x)| + |x| |\nabla {\bf v}_b^{(1)}(x)|)]  \\
& \leq C_0 M_{b}^{ref} + C_1 |M_{b}^{r	ef}|^{2k}.
\end{align*}
\end{proposition}

We postpone the proof of {\bf Proposition \ref{prop_decaymode1}} to {\bf Appendix \ref{app_decaymode1}}. We state the result for ${\bf v}_b$ but a corresponding one also holds for ${\bf v}_a.$
Since $C_0 M^{ref} + C_1 |M^{ref}|^{2k}$ goes to $0$ with $M^{ref}$ whatever the constants $C_0, \, C_1$ and exponent $k,$ we include the control induced by {\bf Proposition \ref{prop_decaymode1}} in the constants $M_{a}^{ref}$ and $M_{b}^{ref}$ and make the slight abuse not to change notations. Hence, below we
fix $\beta_1 > 1/2.$ Remarking that $3 < 2+2\beta_1 < 1+4\beta_1,$ we can then choose 
$\alpha_0$ and $\beta^{(1)}$ to be defined by:
\[
\alpha_0 \in (3, \min (3 + 2\beta_0, 2\phi_0 -1,2+2\beta_1))  \qquad
\beta_1 < 1 < \beta^{(1)} := \dfrac{\alpha_0 - 1}{2} < \dfrac{1}{2} +\beta_1 < 2\beta_1,
\]
and we set:
{
\[
\begin{aligned}
M_{i}^{ref} :=  & \sup_{|x| > 1} |x|^{1+\beta_0}\left( \sum_{\ell=0}^{2} |x|^{\ell}|\nabla^{\ell}{\bf v}_i(x)|\right) \\
& +  \sup_{|x| > 1} \left( |x|^{1+\beta_1} (|{\bf v}_{i,1}(x)| + |x| |\nabla {\bf v}_{i,1}(x)|  + |x|^{1+\beta^{(1)}} (|{\bf v}_{i}^{(1)}(x)| + |x| |\nabla{\bf v}_{i}^{(1)}(x)| )\right).
\end{aligned}
\]}
With these (new) notations, our aim is to prove that, if $M_{a}^{ref} + M_{b}^{ref} + |\mu|$ is sufficiently small then ${\bf u}_a = {\bf u}_b.$ For this, we introduce
\begin{equation*}
	\mathbf{w}=\mathbf{u}_{a}-\mathbf{u}_{b} = \mathbf{v}_{a}-\mathbf{v}_{b}, \qquad q=p_{a}-p_{b},
\end{equation*}%
and, we note that this pair satisfies
\begin{eqnarray}	\label{w equ}
	\left( \mathbf{w}\cdot \nabla \right) \mathbf{u}_{a}+\left( \mathbf{u}%
	_{b}\cdot \nabla \right) \mathbf{w}+\nabla q &=&\Delta \mathbf{w,}
 \\
	\text{div}\mathbf{w} &\mathbf{=}&0,  \notag
\end{eqnarray}
with vanishing boundary conditions on $\partial B.$  We look then for a contraction property on ${\bf w}$ in an appropriate metrics. 
It turns out that we will have to use different norms  for the $0$ mode $w_0,$ the first modes $(w_c^{(n)},w_{s}^{(n)})$ and $(w_c^{(t)},w_s^{(t)})$ and for the remainder ${\bf w}^{(1)}.$ 
To this end, we extract for each of these modes a proper equation in which the other modes are involved in a nonlinear term only. The various results are gathered in the following propositions. Concerning the zero mode, we obtain:
\begin{proposition} \label{prop_0mode}
There exists a constant $C(\alpha_0,\phi_0)$ depending on $\alpha_0$ and $\phi_0$ such that:
\begin{multline}
		\int_{1}^{\infty} r^{\alpha_0} |\partial_r w_0|^{2}\mathrm{d}r +
		\int_{1}^{\infty} \dfrac{|w_0|^2}{r^2} r^{\alpha_0}\mathrm{d}r \\ \leq
		C(\alpha_0,\phi_0) [M_{a}^{ref} + M_{b}^{ref}]^2 
		\left[ 
		\int_{1}^{\infty} \dfrac{|{\bf w}_1|^2}{r^2}
		 r^{\alpha_0 -2 \beta_1}{\rm d}r+
		\int_{1}^{\infty} \dfrac{|{\bf w}^{(1)}|^2}{r^2}
		 r^{\alpha_0 -2\beta^{(1)}}{\rm d}r
		 \right].
\end{multline}
\end{proposition}

As for the first modes, we prove:
\begin{proposition} \label{prop_1mode}
Let $\sigma \in \{s,c\}$ and $d \in \{n,t\}.$ There exists a universal constant $C$ such that:
\begin{multline}
	\sum_{\sigma=c,s}\sum_{d=n,t} \int_{1}^{\infty} r^{2} \left( r^2|\partial_{rr} w_{\sigma}^{(d)}|^{2} + {|\partial_{r} w_{\sigma}^{(d)}|^2} + \dfrac{|w_{\sigma}^{(d)}|^2}{r^2} \right)
		 {\rm d}r  \\ \leq
		C [M_{a}^{ref} + M_{b}^{ref} + |\mu|]^2 
\left[ \int_{1}^{\infty} \dfrac{|{\bf w}_0|^2 + |{\bf w}^{(1)}|^2}{r^2} r^{2(1-\beta_1)}{\rm d}r +  \int_{1}^{\infty} |{\bf w}_1|^2 {\rm d}r \right].
\end{multline}
\end{proposition}

Finally, for the remainder, we have:
\begin{proposition} \label{prop_largemodes}
There exists a constant $C > 0$ such that
\[
\begin{aligned}
& \int_{\mathbb R^2 \setminus \overline{B}} |\nabla {\bf w}^{(1)}|^2 {\rm d}x+ \int_{\mathbb R^2 \setminus \overline{B}} \dfrac{|{\bf w}^{(1)}|^2}{|x|^2} {\rm d}x\\
& \leq C\left( M_{a}^{ref} + M_{b}^{ref}  { + \dfrac{|M_{b}^{ref}|^2}{\eta}} + |\mu|\right) \int_{\mathbb R^2 \setminus \overline{B}} \dfrac{|{\bf w}|^2}{|x|^2}{\rm d}x + \dfrac{\eta}{2} \int_{\mathbb R^2 \setminus \overline{B}} |\nabla {\bf w}_0|^2{\rm d}x,
\end{aligned}
\]
 whatever $\eta \in (0,1].$

\end{proposition}

The proofs of these propositions are detailed in the following subsections. Before going to these technical details, we explain now why they imply our uniqueness result.

\medskip

With our previous choices for $\beta_1,$  $\beta^{(1)}$  and $\alpha_0,$
we apply  {\bf Proposition \ref{prop_0mode}}. We obtain there exists a constant $C_0$
such that:
\begin{align}\notag 
\int_{\mathbb R^2 \setminus \overline{B}} \dfrac{|{\bf w}_{0}|^2}{|x|^2} |x|^{\alpha_0-1} {\rm d}x  \leq& 
		C(\alpha_0,\phi_0) [M_{a}^{ref} + M_{b}^{ref}]^2 
		\left[ 
		\int_{1}^{\infty} \int_{\mathbb T}\dfrac{|{\bf w}_1|^2}{r^2}
		 r^{\alpha_0 -2 \beta_1}{\rm d}r{\rm d}\theta \right.\\
		 & + \left.
		\int_{1}^{\infty}\int_{\mathbb T} \dfrac{|{\bf w}^{(1)}|^2}{r^2}
		 r^{\alpha_0 -2\beta^{(1)}}{\rm d}r{\rm d}\theta
		 \right].\label{eq_mode0-0}
\end{align}
On the other hand, applying {\bf Proposition \ref{prop_1mode}}, we get:
\begin{equation} \label{eq_mode1}
\begin{aligned}
\int_{\mathbb R^2 \setminus \overline{B}} \dfrac{|{\bf w}_1|^2}{|x|} dx \leq & C(\alpha_0,\phi_0) [M_{a}^{ref} + M_{b}^{ref} + |\mu|]^2 
\left[ \int_{1}^{\infty} \int_{\mathbb T}\dfrac{|{\bf w}_0|^2 + |{\bf w}^{(1)}|^2}{r^2} r^{2(1 -\beta_1)}{\rm d}r{\rm d} \theta\right.  
\\
& \left. +  \int_{1}^{\infty}\int_{\mathbb T} \dfrac{|{\bf w}_1|^2}{r^2} r^{2}{\rm d}r{\rm d}\theta \right].
\end{aligned}
\end{equation}
We also apply {\bf Proposition \ref{prop_largemodes}} with $\eta =1$ to yield:
\begin{equation} \label{eq_mode>=2}
\begin{aligned}
\int_{\mathbb R^2 \setminus \overline{B}} \dfrac{|{\bf w}^{(1)}|^2}{|x|^2} {\rm d}x  \leq & C\left(  M_{a}^{ref} + M_{b}^{ref} {+|M_b^{ref}|^2}
  + |\mu|\right) \int_{\mathbb R^2 \setminus \overline{B}} \dfrac{|{\bf w}|^2}{|x|^2}{\rm d}x \\
  & {+ \dfrac{1}{2}\int_{\mathbb R^2 \setminus \overline{B} }|\nabla {\bf w}_0|^2 {\rm d}x }   
\end{aligned}
\end{equation}
\medskip
Finally, combining \eqref{eq_mode0-0}-\eqref{eq_mode1}-\eqref{eq_mode>=2} we have since $\alpha_0 > 1$:
\begin{align*}
 & {\frac 12} \int_{\mathbb R^2 \setminus \overline{B}} \dfrac{|{\bf w}_{0}|^2}{|x|^2} |x|^{\alpha_0-1} {\rm d}x 
  + \int_{\mathbb R^2 \setminus \overline{B}} \dfrac{|{\bf w}_1|^2}{|x|} {\rm d}x  + \int_{\mathbb R^2 \setminus \overline{B}} \dfrac{|{\bf w}^{(1)}|^2}{|x|^2}  {\rm d}x \\
& \leq  C\left(  M_{a}^{ref} + M_{b}^{ref} {+ |M_b^{ref}|^2 } + |\mu|\right)
\left( 
\int_{1}^{\infty}\int_{\mathbb T} \dfrac{|{\bf w}_0|^2}{r^2} r^{2(1 -\beta_1)}{\rm d}r{\rm d}\theta \right.\\
& +
\int_{1}^{\infty} \int_{\mathbb T} \dfrac{|{\bf w}_1|^2}{r^2}
		 (r^{\alpha_0 -2 \beta_1} + r^2){\rm d}r{\rm d}\theta		 
+
	\int_{1}^{\infty}\int_{\mathbb T} \dfrac{|{\bf w}^{(1)}|^2}{r^2}
		 (r^{\alpha_0 -2\beta^{(1)}} + r^{2(1-\beta_1)}){\rm d}r{\rm d}\theta \\
& \left.		
+
	 \int_{\mathbb R^2 \setminus \overline{B}}\dfrac{|{\bf w}|^2}{|x|^2}{\rm d}x 
\right).
\end{align*}
We remark that, by our choice for $\alpha_0$ we have
$2(1-\beta_1) < \alpha_0,$ $\alpha_0-2 \beta_1 < 2,$ and   
$\max(\alpha_0 -2\beta^{(1)},\, 2(1-\beta_1)) \leq 1.$ This entails that:
\[
\begin{aligned}
\int_{1}^{\infty}\int_{\mathbb T} \dfrac{|{\bf w}_0|^2}{r^2} r^{2(1- \beta_1) }{\rm d}r{\rm d}\theta & \leq  \int_{\mathbb R^2 \setminus \overline{B}} \dfrac{|{\bf w}_0|^2}{|x|^2} |x|^{\alpha_0-1}{\rm d}x, \\
 \int_{1}^{\infty}\int_{\mathbb T} \dfrac{|{\bf w}_1|^2}{r^2} (r^{\alpha_0 - 2\beta_1} + r^2) {\rm d}r{\rm d}\theta & \leq \int_{\mathbb R^2 \setminus \overline{B}} \dfrac{|{\bf w}_1|^2}{|x|} {\rm d}x,  \\
\int_{1}^{\infty}\int_{\mathbb T} \dfrac{|{\bf w}^{(1)}|^2}{r^2} (r^{\alpha_0-2\beta^{(1)}} + r^{2(1- \beta_1)}){\rm d}r{\rm d}\theta & \leq   \int_{\mathbb R^2 \setminus \overline{B}} \dfrac{|{\bf w}^{(1)}|^2}{|x|^2}{\rm d}x.
\end{aligned}
\]
We thus obtain that ${\bf w} \equiv 0$ when $\mu$ and $M_{a}^{ref},M_{b}^{ref}$ are sufficiently small.
This completes the proof.

\subsection{Proof of Proposition \ref{prop_0mode}} \label{sec_mode0}
The proof splits into three parts. At first, we obtain a differential equation satisfied by $w_0$ by projecting \eqref{w equ} on the mode $0.$ We extract a linear and a nonlinear part in this equation. We then compute a priori estimates for solutions to the linear problem depending on the nonlinearity. We conclude by showing a priori estimates for the nonlinearity.   

\medskip

We start by extracting a differential equation for $w_0.$
Multiplying \eqref{w equ} by $\mathbf{e}_{\theta}$ and integrating on any circle of radius $r,$ for fixed $r > 1,$ gives that $w_0$ satisfies the equation:
\begin{equation} \label{eq_mode0}
	\dfrac{1}{r^{\phi_0+1}}\partial_{r}[r^{\phi_0+1} \partial_r w_{0}] + \dfrac{%
		\phi_0-1}{r^2} w_0 = {\mathcal{F}_0},
	 \qquad w_0(1) = w_0(\infty) = 0,
\end{equation}
where
\begin{equation*}
	{\mathcal{F}_0} = \dfrac{1}{2\pi} \int{\mathbb T} (\mathbf{w}^{(0)} \cdot \nabla \mathbf{v}%
	_a^{(0)}) \cdot \mathbf{e}_{\theta} \mathrm{d}\theta + \dfrac{1}{2\pi}
	\int_{\mathbb T} ( \mathbf{v}_b^{(0)} \cdot \nabla \mathbf{w}^{(0)}) \cdot \mathbf{e}%
	_{\theta}\mathrm{d}\theta.
\end{equation*}
To reach this equation, we use the following orthogonality relations to simplify nonlinearities:
\[
{\bf v}_{a,0} = v_{a,0,\theta} (r) {\bf e}_{\theta} , \quad \int_{\mathbb T} w^{(0)}_r {\rm d}\theta = \int_{\mathbb T} w^{(0)}_{\theta} {\rm d}\theta =\int_{\mathbb T} \partial_{\theta} {\bf w}^{(0)} \cdot {\bf e}_{\theta} {\rm d}\theta =  0.
\]
We rewrite the right-hand side of this differential equation by remarking that:
\begin{equation*}
	(\mathbf{v}_b^{(0)} \cdot \nabla \mathbf{w}^{(0)}) \cdot\mathbf{e}_{\theta} =
	v_{b,r}^{(0)} \partial_r w_{\theta}^{(0)} + \dfrac{v_{b,\theta}^{(0)}}{r}
	\partial_{\theta} w_{\theta}^{(0)}+ \dfrac{v_{b,\theta}^{(0)} w_r^{(0)}}{r},
\end{equation*}
so that
\begin{equation*}
	\begin{aligned} {\mathcal{F}_0} = & \dfrac{1}{2\pi} \int_{\mathbb T} ({\bf w}^{(0)} \cdot
		\nabla {\bf v}_a^{(0)}) \cdot {\bf e}_{\theta} {\rm d}\theta + \dfrac{1}{2\pi}
		\dfrac{\textrm{d}}{\textrm{d}r} \left[ \int_{\mathbb T} v_{b,r}^{(0)} w_{\theta}^{(0)} {\rm
			d}\theta\right] \\ & + \dfrac{1}{2\pi} \int_{\mathbb T} \left( w_{\theta}^{(0)}
		\dfrac{\partial_r [r v_{b,r}^{(0)}]}{r} -w_{\theta}^{(0)} \partial_r v_{b,r}^{(0)} +
		\dfrac{v_{b,\theta}^{(0)} w_r^{(0)}}{r} \right) {\rm d}\theta. \\& 
		 \end{aligned}
\end{equation*}
Eventually, we conclude that:
\[
\mathcal F^{(0)} = \mathcal{F}_0^{(0)}+ \dfrac{\rm d \mathcal{F}_0^{(1)}}{{\rm d}r},
\]
where:
\begin{equation}
	\left\{
	\begin{aligned}
		\mathcal{F}_0^{(0)} & = \dfrac{1}{2\pi} \int_{\mathbb T} ({\bf w}^{(0)} \cdot \nabla {\bf v}_a^{(0)}) \cdot
		{\bf e}_{\theta} {\rm d}\theta + \dfrac{1}{2\pi} \int_{\mathbb T} \dfrac{
			w_{\theta}^{(0)} v_{b,r}^{(0)} + v_{b,\theta}^{(0)} w_r^{(0)}}{r} {\rm d}\theta, \\
		\mathcal{F}_0^{(1)} & = \dfrac{1}{2\pi}  \left[ \int_{\mathbb T} v_{b,r}^{(0)} w_{\theta}^{(0)}
		{\rm d}\theta\right].
	\end{aligned}
	\right.
\end{equation}
In particular, we have by (\ref{eq_Miref})
\begin{equation} \label{eq_decayF0}
	\lim_{r \to \infty} r^{3+2\beta_0}|\mathcal{F}_0^{(0)}| + \lim_{r\to \infty} r^{2+2 \beta_0}|\mathcal{F}_0^{(1)}| = 0.
\end{equation}
{\bf Proposition \ref{prop_0mode}} then yields as a combination of the two following lemmas:
\begin{lemma} \label{lem_w0}
	Let $\alpha \in (3,\min (2\phi_0-1,2\beta_0+3)).$ Given $(\mathcal{F}_0^{(0)},\, \mathcal{F}_0^{(1)}) \in C([1,\infty)) \times C^1([1,\infty))$ satisfying \eqref{eq_decayF0}, assume that $w \in C^2([1,\infty))$ is a solution to 
	\begin{equation*}
		\dfrac{1}{r^{\phi_0+1}}\partial_{r}[r^{\phi_0+1} \partial_r w] + \dfrac{\phi_0-1}{%
			r^2} w = \mathcal{F}^{(0)}_{0} + \partial_r \mathcal{F}_{0}^{(1)}, \qquad w(1)  = 0,
	\end{equation*}
	satisfying:
	\begin{equation} \label{eq_decayw}
		\sup_{r > 1 } r^{\beta_0+1} |w(r)| + \sup_{r > 1} |r^{\beta_0+2} \partial_r w(r)| < \infty.
	\end{equation}
	Then, we have
	\begin{equation}\label{es_wdr}
		\int_{1}^{\infty} r^{\alpha} |\partial_r w|^{2}\mathrm{d}r +
		\int_{1}^{\infty} \dfrac{|w|^2}{r^2} r^{\alpha}\mathrm{d}r \leq
		C_{\alpha}\int_{1}^{\infty} \left( |\mathcal{F}_0^{(0)}|^2 r^{2} + |\mathcal{F}_{0}^{(1)}|^2
		\right)r^{\alpha} \mathrm{d}r
	\end{equation}
	with a constant $C_{\alpha}$ depending only on $\alpha,\, \phi_0.$
\end{lemma}

We provide now a control of the source term in the norms that we introduced in the previous lemma. This is the content of the following lemma.

\begin{lemma} \label{lem_source0}
	There exists a universal constant $C$ such that, given $\phi_0 > 2$ and $\alpha \in (3,2\phi_0-1)$ there holds 
	\begin{eqnarray*}
		&&\int_{1}^{\infty} \left( |\mathcal{F}_0^{(0)}|^2 r^{2} + |\mathcal F_{0}^{(1)}|^2 
		\right)r^{\alpha} \mathrm{d}r\\
			&\leq& C\left[\sup_{r > 1} |r^{\beta_1+2}\nabla{\bf v}_{a,1}|^2  +  \sup_{r > 1} |r^{\beta_1+1}{\bf v}_{b,1}|^2 \right]
			\int_{1}^{\infty} \dfrac{|{\bf w}_1|^2}{r^2} r^{\alpha - 2\beta_1}{\rm d}r  \\
			&&+ C\left[\sup_{r > 1} |r^{2+\beta^{(1)}}\nabla{\bf v}_{a}^{(1)}|^2  +  \sup_{r > 1} |r^{1+\beta^{(1)}}{\bf v}_{b}^{(1)}|^2 \right]\int_{1}^{\infty} \dfrac{|{\bf w}^{(1)}|^2}{r^2} r^{\alpha - 2\beta^{(1)}}{\rm d}r.
	\end{eqnarray*}
\end{lemma}

We complete this section with proofs for these lemmas. 
 
 \medskip
 
\begin{proof}[Proof of Lemma \ref{lem_w0}]
To obtain our estimate on $w_0,$ we will make repeated use of the following refined Hardy inequalities.
Given $w \in C^1([1,\infty))$ such that $w(1) = 0$ and $\alpha > 1$ there holds:
		\begin{equation}\label{es_wr}
			\int_{1}^{M} \dfrac{|w|^2}{r^2} r^{\alpha}\mathrm{d}r \leq \dfrac{4}{(\alpha -1)^2}
			\int_{1}^{M} |\partial_r w|^2 r^{\alpha}\mathrm{d}r, \qquad \forall \, M > 1.
		\end{equation}
The proof yields by standard integration by parts and Young's inequality, its detail is left to the reader.

\medskip

	Hence, we rewrite our equation 
	\begin{equation*}
		\partial_{rr} w + \dfrac{(\phi_0+1)}{r} \partial_r w + \dfrac{\phi_0-1}{r^2} w =
		\mathcal{F}_0^{(0)} + \partial_r \mathcal{F}_0^{(1)},
	\end{equation*}
	and then
	{\[
		\dfrac{1}{r^{\alpha}}\partial_{r}[r^{\alpha} \partial_r w] + \dfrac{%
			(\phi_0 + 1-\alpha)}{r} \partial_r w+ \dfrac{\phi_0-1}{r^2} w = \mathcal{F}_0^{(0)} +
		\partial_r \mathcal{F}_0^{(1)}.
		\]
		We fix $M$ large,  multiply our equation with $r^{\alpha} w$} and integrate by parts.
	This entails: 
	\begin{eqnarray*}
		&&M^{\alpha} \partial_r w(M) w(M) + \dfrac{(\phi_0+1 - \alpha)}{2} |w(M)|^2 M^{\alpha-1}-\int_{1}^{M} r^{\alpha} |\partial_r w|^{2}\mathrm{d}r \\&
		&+\left(
		(\phi_0-1) -\dfrac{(\phi_0+  1-  \alpha)(\alpha-1)}{2} \right) \int_{1}^{M} |w|^2
		r^{\alpha-2}\mathrm{d}r  \\
		&&= \int_{1}^{M} \left(\mathcal{F}_0^{(0)} + \partial_r \mathcal{F}_0^{(1)}\right) w
		r^{\alpha} \mathrm{d}r.
	\end{eqnarray*}
	Applying \eqref{es_wr},  we infer that: 
	\begin{eqnarray*}
		&&\left[ 1 - \dfrac{4}{(\alpha-1)^2} \left( (\phi_0-1) -\dfrac{(\phi_0 + 1-\alpha)(\alpha-1)}{2} \right) \right] \int_{1}^{M} r^{\alpha} |\partial_r w|^{2}\mathrm{d}%
		r   \\
		&&\leq  M^{\alpha} |\partial_r w(M)| |w(M)|  + |\mathcal{F}_0^{(1)}(M) w(M) M^{\alpha}| +  \dfrac{(\phi_0+1 - \alpha)}{2} |w(M)|^2 M^{\alpha-1}
		\\
		&&+\dfrac{1}{2\eta}\int_{1}^{M} \left(|\mathcal{F}_0^{(0)}|^2r^2 +
		\alpha^2 |\mathcal{F}_0^{(1)}|^2 \right) r^{\alpha} \mathrm{d}r
		+ \dfrac{\eta}{2}
		\int_{1}^{M} \left( 2\dfrac{|w|^2}{r^2} + |\partial_r w|^2 \right)
		r^{\alpha}\mathrm{d}r,
	\end{eqnarray*}
	for arbitrary $\eta >0$ small.  Here, we note that the boundary terms on the second line vanish when $M \to \infty$ under assumptions \eqref{eq_decayF0}-\eqref{eq_decayw} and the condition  $\alpha < 2\beta_0+3$.
	We conclude by combining (\ref{es_wr}) and remarking that, when $\alpha$ is close to $2$ then the factor:
	\[
	\left[ 1 - \dfrac{4}{(\alpha-1)^2} \left( (\phi_0-1) -\dfrac{(\phi_0 + 1-\alpha)(\alpha-1)}{2} \right) \right] > 0,
	\]
	which implies that 
	\[\left( \alpha-1  - \phi_0  \right)^2  - (\phi_0-2)^2<0.
	\]
	This amounts to
	$\alpha \in (3,  2\phi_0-1).$
\end{proof}

\medskip

Now, we give the proof of {\bf Lemma \ref{lem_source0}}.

\begin{proof}[Proof of Lemma \ref{lem_source0}] We provide computations for the term $\mathcal F_0^{(0)}$, the other term is handled in a similar way. For the first term,  we remark that $1$-frequencies and $(1)$-frequencies do not combine in the integrals. Hence, we have: 
\[
\int_{\mathbb T} [{\bf w}^{(0)} \cdot \nabla {\bf v}_a^{(0)} ] \cdot {\bf e}_{\theta}{\rm d}\theta  = 
\int_{\mathbb T} [{\bf w}_1 \cdot \nabla {\bf v}_{a,1} ] \cdot {\bf e}_{\theta}{\rm d}\theta 
+
\int_{\mathbb T} [{\bf w}^{(1)} \cdot \nabla {\bf v}_a^{(1)} ] \cdot {\bf e}_{\theta}{\rm d}\theta,
\]
and then the separate bounds:
\[
\begin{aligned}
\int_{1}^{\infty}\left| \int_{\mathbb T} [{\bf w}_1 \cdot \nabla {\bf v}_{a,1} ] \cdot {\bf e}_{\theta}{\rm d}\theta  
\right|^2 r^{\alpha+2}{\rm d}r
&  \leq
C\sup_{r > 1} |r^{\beta_1+2}\nabla{\bf v}_{a,1}|^2 \int_{1}^{\infty} \dfrac{|{\bf w}_{1}|^2}{r^2} r^{\alpha-2\beta_1}{\rm d}r,
\\
\int_{1}^{\infty}\left| \int_{\mathbb T} [{\bf w}^{(1)} \cdot \nabla {\bf v}_{a}^{(1)} ] \cdot {\bf e}_{\theta}{\rm d}\theta  
\right|^2 r^{\alpha+2}{\rm d}r
& \leq
C\sup_{r>1} |r^{2+\beta^{(1)}}\nabla{\bf v}_{a}^{(1)}|^2 \int_{1}^{\infty} \dfrac{|{\bf w}^{(1)}|^2}{r^2} r^{\alpha-2\beta^{(1)}}{\rm d}r.
\end{aligned}
\]
For the other one, we have similarly:
\[
 \int_{\mathbb T} 
			 (w_{\theta}^{(0)} v_{b,r}^{(0)} + v_{b,\theta}^{(0)} w_r^{(0)} ) {\rm d}\theta
= 
\int_{\mathbb T} 
			 (w_{\theta,1} v_{b,r,1} + v_{b,\theta,1} w_{r,1}) {\rm d}\theta
			 +
\int_{\mathbb T} 
			 (w_{\theta}^{(1)} v_{b,r}^{(1)} + v_{b,\theta}^{(1)} w_r^{(1)} ) {\rm d}\theta,
\]
and 
\begin{align*}
& \int_{1}^{\infty}\left| \int_{\mathbb T} \dfrac{
			w_{\theta}^{(0)} v_{b,r}^{(0)} + v_{b,\theta}^{(0)} w_r^{(0)}}{r} {\rm d}\theta\right|^2 r^{2+\alpha}{\rm d}r
			\\
			&\leq  
\sup_{r > 1} |r^{\beta_1+1}{\bf v}_{b,1}|^2 \int_{1}^{\infty} \dfrac{|{\bf w}_{1}|^2}{r^2} r^{\alpha-2\beta_1}{\rm d}r
+
\sup_{r>1} |r^{1+\beta^{(1)}}{\bf v}_{b}^{(1)}|^2 \int_{1}^{\infty} \dfrac{|{\bf w}^{(1)}|^2}{r^2} r^{\alpha-2\beta^{(1)}}{\rm d}r.
\end{align*}
This completes the proof.
\end{proof}

\subsection{Proof of Proposition \ref{prop_1mode}} \label{sec_mode1}
In this section,  we compute an a priori estimate for
\[
{\bf w}_1 =  \left( w_c^{(n)}(r) \cos\theta + w_s^{(n)} \sin\theta\right) {\bf e}_{r} + \left( w_c^{(t)}(r) \cos\theta + w_s^{(t)} \sin\theta\right) {\bf e}_{\theta},
\] 
where, by construction:
\[
\begin{aligned}
	w_c^{(n)}(r) &= \dfrac{1}{\pi}\int_{\mathbb T} w_r(r,\theta) \cos\theta{\rm d}\theta\,, \quad & w_s^{(n)}(r) & =  \dfrac{1}{\pi}\int_{\mathbb T} w_r(r,\theta) \sin\theta{\rm d}\theta\,, \\
	w_c^{(t)}(r) &= \dfrac{1}{\pi} \int_{\mathbb T} w_{\theta} (r,\theta)\cos\theta{\rm d}\theta\,, \quad & w_s^{(t)}(r) & =  \dfrac{1}{\pi}\int_{\mathbb T} w_{\theta}(r,\theta) \sin\theta{\rm d}\theta\,. \\
\end{aligned}
\]
The proof follows the same structure as for the zero mode. Firstly, we identify an ode satisfied by the first modes in which we separate a linear part and nonlinearities. We focus on obtaining fine a priori estimates on the underlying linear odes that we combine with estimates on nonlinearities in the same spaces. We provide a proof for $w_c^{(n)}$ and $w_s^{(t)}.$ Similar computations (exchanging $sin$ and $cos$ in the following method)  would provide a proof for the other terms.

\subsubsection{Derivation of a fourth order equation}
Plugging the expansion \eqref{eq_expansionu} into \eqref{w equ} entails that
\begin{equation} \label{eq_stokesrevisited}
	\left\{
	\begin{aligned} \Delta {\bf w} + \phi_0 \left( \dfrac{w_{\theta}}{r^2} {\bf
			e}_{\theta} - \dfrac{w_r }{r^2}{\bf e}_r + \dfrac{1}{r} \partial_{r} {\bf
			w}\right) - \nabla q & = {\bf F},\\ \dfrac{1}{r} \partial_r [r w_r]
		+\dfrac{1}{r} \partial_{\theta} w_{\theta} & = 0, \end{aligned}
	\right.
\end{equation}
with 
\begin{equation*}
	\mathbf{F} = \mathbf{w} \cdot \nabla \mathbf{v}_a + \mathbf{v}_{b} \cdot
	\nabla \mathbf{w} + \dfrac{\mu}{r^2} \partial_{\theta} \mathbf{w} - \dfrac{%
		\mu}{r^2} \left( w_{\theta} \mathbf{e}_r + w_r \mathbf{e}_{\theta}\right).
\end{equation*}
Identifying cosine and sine modes in the divergence-free constraint ensures then that:
\begin{equation}  \label{eq_continuity_mode1}
	\left\{
	\begin{aligned}
		\dfrac{\text{d}}{\text{d}r}[ r w_{c}^{(n)}] + {w_{s}^{(t)} } = 0, \\
		\dfrac{\text{d}}{\text{d}r}[ r w_{s}^{(n)}] - {w_{c}^{(t)} } = 0.
	\end{aligned}
	\right.
\end{equation}
We note that this latter equation implies that we have the further boundary condition in $r=1:$
\begin{equation} \label{eq_bc1}
	\left\{
	\begin{aligned}
		& \partial_r w_c^{(n)}(1) = -w_c^{(n)}(1) - w_s^{(t)}(1) = 0,  \\
		& \partial_r w_s^{(n)}(1) = w_c^{(t)}(1) - w_s^{(n)}(1) = 0.   \\
	\end{aligned}
	\right.
\end{equation}

We dot multiply then the first equation of \eqref{eq_stokesrevisited} with $\cos\theta\ \mathbf{e}_r$ and integrate on
circles. We obtain: 
\begin{equation*}
	LHS_{\Delta} + \phi_0 LHS^{\prime}- LHS_q = RHS_c^{(n)},
\end{equation*}
with
\begin{equation*}
	\begin{aligned} LHS_{\Delta} & = \int_{\mathbb T} \left( \dfrac{1}{r}
		\partial_r[r \partial_r {\bf w}] \cdot \cos\theta {\bf e}_r
		\right){\rm d}\theta + \dfrac{1}{r^2} \int_{\mathbb T} \partial_{\theta \theta
		} {\bf w} \cdot \cos\theta {\bf e}_r \rm{d}\theta ,\\
LHS' &= \int_{\mathbb T} \left( \dfrac{1}{r} \partial_r {\bf
			w} \cdot \cos\theta {\bf e}_r - \dfrac{w_r}{r^2} \cos\theta \right){\rm
			d}\theta ,\\
	LHS_{q} & = \int_{\mathbb T} \partial_r q \cos\theta\mathrm{d}\theta,\\
	RHS^{(n)}_c & = \int_{\mathbb T} \left( \dfrac{\partial_r (r
			w_{r}{v}_{a,r} + rv_{b,r} w_{r})}{r} + \dfrac{\mu {\bf e}_r \cdot  \partial_{\theta} {\bf
				w}}{r^2} - \dfrac{\mu w_{\theta}}{r^2} \right) \cos\theta
		\rm{d}\theta \\
		& \quad  +\int_{\mathbb T}\dfrac{\partial_{\theta} (w_{\theta}{\bf
				v}_{a}+ v_{b,\theta}{\bf w} )\cdot {\bf e}_r }{r} \cos \theta {\rm d}\theta.	
	\end{aligned} 
\end{equation*}
Performing straightforward integration by parts, this entails:
\[
\dfrac{1}{r} \dfrac{\text{d}}{\text{d}r} \left[ r \dfrac{\text{d}}{\text{d}r}
w^{(n)}_{c} \right] + \phi_0 \dfrac{\text{d}}{\text{d}r} \left[ \dfrac{w_c^{(n)}}{r}%
\right] - \dfrac{2}{r^2} \left( w_c^{(n)} + w_s^{(t)} \right) = RHS^{(n)}_c
+ \int_{\mathbb T}  \partial_r q \cos\theta\mathrm{{d}\theta,}
\]
where
\begin{equation*}
	\begin{aligned}  & 
RHS^{(n)}_c  =& \dfrac{1}{r}\dfrac{\textrm{d}}{\textrm{d}r}\left[ r F^{(n)}_c\right] + G^{(n)}_c,
	\end{aligned}
\end{equation*}
with 
\begin{equation*}
	\begin{aligned} F^{(n)}_c& = \int_{\mathbb T} \left( w_{r}{v}_{a,r} + v_{b,r} w_r
		\right) \cos\theta {\rm d}\theta, \\G^{(n)}_c &= \int_{\mathbb T} \left(
		\dfrac{\mu}{r^2} w_{r} \sin\theta + \dfrac{(w_{\theta} v_{a,r} +
			v_{b,\theta}w_{r})}{r}\sin\theta \right) {\rm d}\theta \\ &-
		\int_{\mathbb T} \left( \dfrac{(w_{\theta} v_{a,\theta} + v_{b,\theta}
			w_{\theta} )}{r}\cos\theta + \dfrac{2\mu}{r^2} w_{\theta}\cos\theta
		\right) {\rm d}\theta. \end{aligned}
\end{equation*}
To cancel the pressure term, we compute an equivalent equation by
dot multiplying with $\sin\theta  \ \mathbf{e}_{\theta}.$
In a similar fashion, we have
\[
\dfrac{1}{r} \dfrac{\text{d}}{\text{d}r} \left[ r \dfrac{\text{d}}{\text{d}r}
w^{(t)}_{s} \right] + \dfrac{\phi_0}{r^2}\dfrac{\text{d}}{\text{d}r}
[rw_s^{(t)}] - \dfrac{2}{r^2} \left[ w^{(t)}_{s} + w^{(n)}_c\right] =
RHS^{(t)}_s -\int_{\mathbb T} \dfrac{q}{r} \cos\theta\mathrm{{d}\theta},
\]
where
\[
RHS^{(t)}_s =
		\dfrac{1}{r} \dfrac{\textrm{d}}{\textrm{d}r} \left[ r F_s^{(t)} \right] + G_s^{(t)},
\]
with 
\begin{equation*}
	\begin{aligned} 
		F_s^{(t)} &= \int_{\mathbb T} (w_{r}{v}_{a,\theta} + v_{b,r}
		w_{\theta}) \sin \theta {\rm d}\theta, \\ G_s^{(t)} &= \int_{\mathbb T}\left(
		\dfrac{(w_{\theta} v_{a,r} + v_{b,\theta }w_r)}{r} \sin\theta -
		\dfrac{(w_{\theta} v_{a,\theta} + v_{b,\theta} w_{\theta} )}{r} \cos\theta
		- \dfrac{\mu}{r^2} w_{\theta} \cos\theta\right) {\rm d}\theta. \end{aligned}
\end{equation*}
We point out that, with our assumptions on the decay of ${\bf v}_a$ and ${\bf v}_b,$ we have the expected decay of source terms:
\begin{equation} \label{eq_decaysource1}
\sup_{r > 1} r^{2+\beta_0} \left[\left( |F_l^{(k)}| + r |G_l^{(k)}|\right) + r \left( |\partial_r F_l^{(k)}| + r |\partial_r G_l^{(k)}|\right) \right] < \infty, 
\end{equation}
for $(l,k) \in  \{(s,t),(c,n) \}$.
So, we infer the following system:
\begin{equation*}
	\left\{ \begin{aligned} \dfrac{1}{r} \dfrac{\textrm{d}}{\textrm{d}r} \left[
		r \dfrac{\textrm{d}}{\textrm{d}r} w^{(n)}_{c} \right] + \phi_0
		\dfrac{\textrm{d}}{{\textrm d}r} \left[ \dfrac{w^{(n)}_c}{r}\right] -
		\dfrac{2}{r^2} \left( w^{(n)}_c + w^{(t)}_s \right)& = RHS^{(n)}_c + \dfrac{\textrm{d}}{\textrm{d}r} \int_{\mathbb T}
		q \cos\theta {\rm d}\theta, \\ \dfrac{1}{r} \dfrac{\textrm{d}}{\textrm{d}r} \left[ r
		\dfrac{\textrm{d}}{\textrm{d}r} w^{(t)}_{s} \right] +\dfrac{\phi_0}{r^2}
		\dfrac{\textrm{d}}{\textrm{d}r} [rw^{(t)}_s] - \dfrac{2}{r^2} \left( w^{(t)}_{s} +
		w^{(n)}_c\right) & = RHS^{(t)}_s -\int_{\mathbb T} \dfrac{q}{r} \cos\theta {\rm d}\theta, \\
		\dfrac{\textrm{d}}{\textrm{d}r}[ r w^{(n)}_{c}] + {w^{(t)}_{s} } & = 0. \end{aligned} %
	\right.
\end{equation*}
To treat this system, we skip exponents in the following lines for legibility. So, we apply $\partial_r (r \cdot)$ to the second equation and combine with
the first one. This entails 
\begin{equation} \label{prime eqns}
	-r^2 w_c^{(4)} - (6+\phi_0) rw_c^{(3)} - (3 + 4\phi_0) w_c^{^{\prime\prime}} + 3 
	\dfrac{w_c^{\prime}}{r} = RHS^{(n)}_c + \dfrac{\text{d}}{\text{d}r} [r RHS^{(t)}_s ],
\end{equation}
with the boundary conditions
\begin{equation*}
	\partial_r w_c (1) = w_c(1) = 0, \qquad \lim_{r \to \infty} w_c(r) = \lim_{r
		\to \infty} \partial_r w_c(r) =0.
\end{equation*}

To complete the proof of {\bf Proposition \ref{prop_1mode}}, we compute now a priori estimates for the model equation \eqref{prime eqns} and derive corresponding control on the nonlinearities. {\bf Proposition \ref{prop_1mode}} yields as a straightforward combination of the following two lemmas:  
\begin{lemma} \label{lem_w1}
	Let $\beta_0 > 0$ and $\phi_0 > 2.$  Given $({F^{(n)}}, F^{(t)}) \in C^2([1,\infty))  $ and $(G^{(n)},G^{(t)}) \in C^1([1,\infty))$ satisfying \eqref{eq_decaysource1}, assume that $w \in C^4([1,\infty))$ is a solution to 
\begin{equation} \label{eq_model1}
	-r^2 \partial_{r}^4w - (6+\phi_0) r \partial_{r}^3 w - (3 + 4\phi_0) \partial_{r}^2 w + 3 
	\dfrac{\partial_{r} w}{r} = RHS^{(n)} + \partial_r [r RHS^{(t)} ],
\end{equation}
where
\[
	RHS^{(n)} = \dfrac{1}{r} \partial_r  \left[ r F^{(n)} \right] + G^{(n)}, \qquad 
	RHS^{(t)} = \dfrac{1}{r}\partial_r  \left[ r F^{(t)} \right] + G^{(t)}
\]
satisfy \eqref{eq_decaysource1}.
If this solution satisfies further the boundary condition:
\[
	\partial_r w (1) = w(1) = 0
\]
and decay properties:
\[
\begin{aligned}
\sup_{r > 1} r^{1+\beta_0} |w(r)| +  r^{2+\beta_0} |\partial_r w(r)| +  r^{3+\beta_0} |\partial_{rr} w(r)| +  r^{4+\beta_0} |\partial_{rrr} w(r)| < \infty,
\end{aligned}
\]
then, there exists a constant $C(\phi_0)$ depending only on $\phi_0$ for which
	\begin{eqnarray*}
		&&\sum_{i=0}^2 \int_{1}^{\infty} |\partial_{r}^{i}w(r)|^2r^{2i}{\rm d}r  \\
		&\leq & C(\phi_0) 
		\int_1^{\infty} \left[ |G^{(n)}(r)|^2 r^2 + |F^{(n)}(r)|^2 +|G^{(t)}(r)|^2 r^2 + |F^{(t)}(r)|^2  \right]r^{2}{\rm d}r.
	\end{eqnarray*}
\end{lemma}
\begin{lemma} \label{lem_source1}
There exists a universal constant $C$   such that there holds:	\begin{eqnarray*}
	&&	\int_{1}^{\infty} \left( |G_{c}^{(n)}(r)|^2 + |G_{s}^{(t)}(r)|^2\right) r^{4}{\rm d}r + \int_{1}^{\infty} \left( |F_{c}^{(n)}(r)|^2+ |F_{s}^{(t)}(r)|^2\right)  r^{2}{\rm d}r \\
			& \leq& C\left( (\mu^2 + \sup_{r>1} |r^{1+\beta_0}  {\bf v}_a|^2 + \sup_{r>1} |r^{1+\beta_0}  {\bf v}_b|^2 )   \int_{1}^{\infty}
			\int_{\mathbb T} |{\bf w}_1(r,\theta)|^2  {\rm d}r{\rm d}\theta \right. \\ &&
			\left. + \sup_{r>1} (|r^{1+\beta_1} {\bf v}_{a}^{(0)}|^2 + |r^{1+\beta_1}{\bf v}_{b}^{(0)}|^2)
			\int_{1}^{\infty} \int_{\mathbb T} \dfrac{|{\bf w}_0(r,\theta)|^2 + |{\bf w}^{(1)}(r,\theta)|^2}{r^2} r^{2(1-\beta_1)}{\rm
				d}r {\rm d}\theta \right).
	\end{eqnarray*}
\end{lemma}

Proofs for these lemmas are provided in the next  subsections.

\subsubsection{Proof of Lemma \ref{lem_w1}}
We multiply (\ref{eq_model1}) with $r^{2} w$ and integrate between $r=1$ and $r=M$ arbitrary large. This entails 
\begin{align} \notag 
	- \int_{1}^{M} \partial_r^{4} w \, w r^{4} {\rm d}r & - (6+\phi_0) \int_{1}^{M} \partial_r^{3} w \, w r^{3} {\rm d}r - (3+4\phi_0) \int_{1}^{M} \partial_r^{2} w \, w r^{2} {\rm d}r + 3 \int_{1}^{M} \partial_r w \, w r  {\rm d}r \\
	& = \int_{1}^{M} RHS^{(n)} w r^{2} {\rm d}r  + \int_{1}^{M} \partial_r [r RHS^{(t)}] w r^{2} {\rm d}r. \label{eq_Is}
\end{align}

Because of the decay of $w$ at infinity we can integrate by parts the left-hand side $LHS$ of this identity to yield:
\[
LHS= \varepsilon(M)  - \int_{1}^{\infty} |\partial_r^{2} w(r)|^2 r^{4}\mathrm{d}r -  \dfrac{\phi_0}{2}\int_{1}^{\infty} |\partial_r w(r)|^2 r^{2}\mathrm{d}r - \left( \phi_0 - \dfrac{3}{2}\right) \int_{1}^{\infty} |w(r)|^2 \mathrm{d}r,
\]
where $\varepsilon(M) \to 0$ when $M \to \infty.$ The main point here is that all coefficients of integrals are strictly less than $-1/2$ since $\phi_0 > 2.$ 

\medskip

For the right-hand side $RHS$ of (\ref{eq_Is}), we recall the decay property \eqref{eq_decaysource1} that entails in particular that $r^{3+\beta_0}RHS^{(t)} \to 0$ at infinity. We obtain then, by integrating by parts a function $\varepsilon(M)$ vanishing while $M \to \infty$ such that:
\begin{equation*}
	\begin{aligned} RHS	& =\varepsilon(M)
		 - \int_{1}^{M} F^{(n)}(r)
		[\partial_r w(r) r^{2} +  w(r) r] {\rm d}r \\ 
		&\quad  +  \int_{1}^{M} G^{(n)}(r) w(r)r^{2} {\rm d}r -
		\int_{1}^{M} G^{(t)}(r) \left[ \partial_r w(r) r^{3} + 2 w(r)
		r^{2} \right] {\rm d}r \\ &
		\quad  + \int_{1}^{M} F^{(t)}(r) \left[ \partial_{rr}w(r)
		r^{3} + 4 \partial_r w(r) r^{2} + 2  w(r)
		r \right] {\rm d}r. \end{aligned}
\end{equation*}
We now estimate the various terms in $RHS$ via standard Minkowski inequality yielding a universal constant $C$ such that  for arbitrary $\eta \in (0,1):$
\begin{multline*}
	|RHS|
	\leq  \varepsilon(M) + \dfrac{\eta}{2} 
	\left( 
	\int_{1}^{\infty} |\partial_r^{2} w(r)|^2 r^{4}\mathrm{d}r +  \int_{1}^{\infty} |\partial_r w(r)|^2 r^{2}\mathrm{d}r + \int_{1}^{\infty} |w(r)|^2 \mathrm{d}r.
	\right)
	\\ +
	\dfrac{C}{\eta}\int_1^{\infty} \left[ |G^{(n)}(r)|^2 r^2 + |F^{(n)}(r)|^2 +|G^{(t)}(r)|^2 r^2 + |F^{(t)}(r)|^2  \right]r^{2}{\rm d}r.
\end{multline*}
Choosing $\eta < 1/4$ and letting  $M$ go to infinity concludes the proof.

\subsubsection{Proof of Lemma \ref{lem_source1}}
We focus on the $G_{c}^{(n)}$ term. All the other quantities are handled in the same way. 
We first note that:
\[
G_{c}^{(n)} =\pi \mu \dfrac{w_{s}^{(n)} +2 w_{c}^{(t)}}{r^2} + \dfrac{\tilde{G}_{c}^{(n)}}{r},
\]
where $\tilde{G}_{c}^{(n)}$ is a combination of integrals of the form: 
\[
\int_{\mathbb T} ({\bf v} \cdot {\bf e}) \, ({\bf w} \cdot {\bf \tilde{e}})\, \sigma(\theta) {\rm d}\theta
\]
with $({\bf e},\tilde{\bf e}) \in \{{\bf e}_r,{\bf e}_{\theta}\}$ and ${\bf v} \in \{{\bf v}_a,{\bf v}_b\}$ and
$\sigma(\theta) \in \{\cos\theta,\sin\theta \}.$ We expand ${\bf w} = {\bf w}_0 + {\bf w}_1 + {\bf w}^{(1)}$ and similarly with ${\bf v}$ and we recall that the radial and tangential components of ${(1)}$-modes are orthogonal to constant and first modes. We obtain:
\[
\begin{aligned}
\int_{\mathbb T} ({\bf{v} \cdot {\bf e}) ({\bf w} \cdot {\bf \tilde{e}}}) \sigma(\theta) {\rm d}\theta
 = & \int_{\mathbb T} {(\bf v}_{1} \cdot {\bf e}) ({\bf w}_0 \cdot {\bf \tilde{e}}) \sigma(\theta) {\rm d} \theta 
+ \int_{\mathbb T} {(\bf v} \cdot {\bf e}) ({\bf w}_1 \cdot {\bf \tilde{e}}) \sigma(\theta) {\rm d}\theta \\
&+  \int_{\mathbb T} {(\bf v}^{(0)} \cdot {\bf e}) ({\bf w}^{(1)} \cdot {\bf \tilde{e}}) \sigma(\theta) {\rm d}\theta,
\end{aligned}
\]
and thus
\[
\left|\int_{\mathbb T} ({\bf{v} \cdot {\bf e}) ({\bf w} \cdot {\bf \tilde{e}}}) \sigma(\theta) {\rm d}\theta \right|
\leq  C\left[  (\sup_{r>1} |r^{1+\beta_1}{\bf v}^{(0)}|) \dfrac{|{\bf w}_{0}| + |{\bf w}^{(1)}|}{r^{1+\beta_1}}  + (\sup_{r>1} |r^{1+\beta_0}  {\bf v}| )  \dfrac{|{\bf w}_1|}{r^{1+\beta_0}}    \right].
\]
Consequently, we obtain:
\[
\begin{aligned}
\int_{1}^{\infty} |G_{c}^{(n)}|^2 r^{4} {\rm d}r 
\leq & C( |\mu|^2 + \sup_{r>1} |r^{1+\beta_0}  {\bf v}|^2 )  \int_{1}^{\infty}   \dfrac{|{\bf w}_1|^2}{r^2} r^{2}{\rm d}r 
\\
& 
+ C \sup_{r > 1}  |r^{1+\beta_1}{\bf v}^{(0)}|^2
\int_{1}^{\infty}   \dfrac{|{\bf w}_0|^2 + |{\bf w}^{(1)}|^2}{r^2} r^{2(1- \beta_1)}{\rm d}r .
\end{aligned}
\]

\subsection{Proof of Proposition \ref{prop_largemodes}} \label{sec_mode>=2}

In this section, we control higher order frequencies via a standard multiplier argument. To obtain energy estimate, we want to multiply (\ref{w equ}) by $\mathbf{w}^{(0)}$ and integrate by
parts on $B(0,M) := B_M$ with $M$ arbitrary large. 
To avoid external boundary terms, we prefer a truncation procedure. Namely, we know that ${\bf w}^{(0)} = \nabla^{\bot} \varphi^{(0)}$ where 
\[
\sup_{|x| > 1}\left( |x|^{\beta_0} |\varphi^{(0)}(x)| + |x|^{1+\beta_0} |\nabla \varphi^{(0)}| + |x|^{2+\beta_0} |\nabla^2 \varphi^{(0)}| \right) < \infty.
\]
We consider then $\zeta_{M} \in C^{\infty}(\mathbb R^2)$ a truncation function satisfying
\[
\mathds{1}_{B_M} \leq \zeta_M \leq \mathds{1}_{B_{2M}} \text{  with } |\nabla^{k} \zeta_M| \leq \dfrac{C(k)}{M^{k}}, \quad \forall \, k \in \mathbb N.
\]
We denote ${\bf w}^{(0)}_{M} := \nabla^{\bot} \zeta_M \varphi^{(0)}.$ Combining the decay of ${\bf w}^{(0)}$ and $\zeta_M$ we have that there is a constant $C$ such that outside $B_M$ there holds:
\[
|{\bf w}_M^{(0)}| + |x| |\nabla {\bf w}_M^{(0)}| \leq  \dfrac{C}{|x|^{1+\beta_0}}. 
\]
Consequently, one obtains:%
\begin{eqnarray*}
\int_{B_M \backslash \overline{B}}\mathbf{w}^{(0)}\cdot \left( \mathbf{w}\cdot \nabla
	\right) \mathbf{u}_{a}{\rm d}x + \int_{B_M \backslash \overline{B}}\mathbf{w}^{(0)}\cdot
	\left( \mathbf{u}_b\cdot \nabla \right) \mathbf{w}{\rm d}x +\int_{B_M \backslash 
		\overline{B}} \nabla \mathbf{w} : \nabla \mathbf{w}^{(0)} {\rm d}x 
	= - \varepsilon(M),
\end{eqnarray*}
where we used that $\mathbf{w}^{(0)} =0$ on $\partial B$
and we denoted:
\begin{align*}
\varepsilon(M) =&  \int_{B_{2M} \backslash \overline{B_M}}\mathbf{w}^{(0)}_M \cdot \left( \mathbf{w}\cdot \nabla
	\right) \mathbf{u}_{a}{\rm d}x + \int_{B_{2M} \backslash \overline{B_M}}\mathbf{w}^{(0)}_M\cdot
	\left( \mathbf{u}_b\cdot \nabla \right) \mathbf{w}{\rm d}x \\
	& +\int_{B_{2M} \backslash 
		\overline{B_M}} \nabla \mathbf{w} : \nabla \mathbf{w}^{(0)}_M{\rm d}x .
\end{align*}
Thanks to the decay of ${\bf u}_{a}, {\bf u}_{b},$ $\bf{w}$ and ${\bf w}^{(0)}_M$, we infer that there is a constant $C$ for which:
\[
|\varepsilon(M)| \leq  \dfrac{C}{M^{2+2\beta_0}} 
\]
goes to $0$ when $M \to \infty.$ Like in the previous computations, we keep the same symbol below $\varepsilon(M)$ for a vanishing function that may change between lines.
Concerning the left-hand side, passing to radial coordinates, we remark that we have the following orthogonality properties:
\begin{equation*}
\int_{B_{M} \setminus \overline{B}} \mathbf{v}^{(0)}\cdot(\xi_r(r)\mathbf{e}_r+\xi_\theta(r)\mathbf{e}_\theta){\rm d}x=0, \qquad \int_{B_M \setminus \overline{B}} \nabla \mathbf{v}^{(0)}:\nabla \nabla^{\bot} \xi(r){\rm d}x=0,
\end{equation*}
for arbitrary $\xi_r(r), \, \xi_\theta(r), \, \xi(r)\in C_c^{\infty}([1,\infty)). $
This entails:
\[
\int_{B_M \setminus \overline{B}} \nabla {\bf w} : \nabla {\bf w}^{(0)}{\rm d}x = 
\int_{B_{M}\setminus \overline{B}} |\nabla {\bf w}^{(0)}|^2{\rm d}x.
\]
 Eventually, we infer that:
\begin{equation*}
	\int_{B_M \setminus \overline{B}} |\nabla \mathbf{w}^{(0)}|^2{\rm d}x + \int_{B_M \setminus \overline{B}} 
	\mathbf{w}^{(0)} \cdot \left[ \left( \mathbf{w} \cdot \nabla \right) \mathbf{u}%
	_{a}+ \mathbf{u}_b \cdot \nabla \mathbf{w} \right] {\rm d}x = \varepsilon(M).
\end{equation*}
In this identity, we remark that ${\bf u}_a$ and ${\bf u}_b$
contain possibly large terms. Since we have no performing 2D Hardy inequality at-hand to control this term (recall that $\nabla {\bf u}_a$ decays like $1/|x|^2$),  we need to extract some signed terms from this quantity to handle possible remainders. For this we remark first again that, since ${\bf w}$ vanishes on the boundaries and decay fast at infinity, we have by integration by parts:
\begin{equation*}
	\int_{B_M \setminus \overline{B}} \mathbf{w}^{(0)} \cdot \left[ \mathbf{u}_b \cdot \nabla 
	\mathbf{w}^{(0)} \right] {\rm d}x = \int_{\mathbb T} u_{b,r}  (M,\theta) \dfrac{|{\bf w}^{(0)}(M,\theta)|^2}{2} M{\rm d}\theta =  \varepsilon(M).
\end{equation*}
Also, since ${\bf w}_0 = w_0(r) {\bf e}_{\theta}$ there holds 
\begin{equation*}
	\left( -\dfrac{\phi_0}{r} \mathbf{e}_r + \dfrac{\mu}{r}\mathbf{e}_{\theta}\right)
	\cdot \nabla \mathbf{w}_0  = -\dfrac{\phi_0}{r} \partial_r w_0(r){\bf e}_{\theta} - \dfrac{\mu}{r^2} w_0(r)  {\bf e}_{r}.
\end{equation*}
 Hence, the orthogonality properties mentioned above entail that:
\begin{equation*}
	\int_{B_M \setminus \overline{B}} \mathbf{w}^{(0)} \cdot \left[ \mathbf{u}_b \cdot \nabla 
	\mathbf{w} \right] {\rm d}x = \int_{B_M \setminus \overline{B}} \mathbf{w}^{(0)} \cdot \left[ 
	\mathbf{v}^{(0)}_b \cdot \nabla \mathbf{w}_0 \right] {\rm d}x + \varepsilon(M).
\end{equation*}
With a similar remark, we compute: 
\begin{equation*}
	\int_{B_M \setminus B} \mathbf{w}^{(0)} \cdot \left[ \mathbf{w}_0 \cdot \nabla
	\left( -\dfrac{\phi_0}{r} \mathbf{e}_r + \dfrac{\mu}{r}\mathbf{e}_{\theta}
	{+ {\bf v}_{a,0}} \right) %
	\right] {\rm d}x= 0.
\end{equation*}
So that we end up with the formula: 
\begin{align} \label{eq_rwfinal}
\int_{B_M \setminus \overline{B}} |\nabla \mathbf{w}^{(0)}|^2 \mathrm{d}x & + \int_{B_M \setminus \overline{B}} 
	\mathbf{w}^{(0)} \cdot \left[ \left( \mathbf{w}^{(0)} \cdot \nabla \right) \mathbf{u}_{a} \right] {\rm d}x \notag \\
	& + \int_{B_M \setminus \overline{B}} 
	\mathbf{w}^{(0)} \cdot \left[ (\mathbf{w}_{0} \cdot \nabla) \mathbf{v}_a^{(0)} + \mathbf{v}^{(0)}_b \cdot \nabla 
	\mathbf{w}_0 \right] {\rm d}x 
	= \varepsilon(M).
\end{align}

We extract now the leading term in the left-hand side of \eqref{eq_rwfinal} by expanding ${\bf u}_a.$ Indeed, for the second term in the nonlinearity, we observe that: 
\begin{equation*}
	\mathbf{w}^{(0)} \cdot \left( \mathbf{w}%
	^{(0)}\cdot \nabla \right) \mathbf{u}_{a}{\rm d}x= \left( w^{(0)}_{r}\mathbf{e}_{r}+w^{(0)}_{\theta }\mathbf{e}_{\theta }\right)
	\cdot \left( w^{(0)}_{r}\partial _{r}+w^{(0)}_{\theta }\frac{\partial _{\theta }}{r}%
	\right) \mathbf{u}_{a}{\rm d}x,
\end{equation*}
where
\begin{eqnarray*}
	\partial _{r}\mathbf{u}_{a} &=&\frac{\phi_0 }{r^{2}}\mathbf{e}_{r} - \dfrac{\mu%
	}{r^2}\mathbf{e}_{\theta} +\partial _{r}\mathbf{v}_{a}, \\
	\frac{\partial _{\theta }\mathbf{u}_{a}}{r} &=&-\frac{\phi_0 }{r^{2}}\mathbf{e}%
	_{\theta } - \dfrac{\mu}{r^2} \mathbf{e}_{r} +\frac{\partial _{\theta }%
		\mathbf{v}_{a}}{r}.
\end{eqnarray*}%
Consequently, there holds: 
\begin{eqnarray*}
	\int_{B_M \backslash \overline{B}}\mathbf{w}^{(0)}\cdot \left( \mathbf{w}^{(0)} \cdot
	\nabla \right) \mathbf{u}_{a}{\rm d}x &=&\int_{B_M \backslash \overline{B}} \frac{\phi_0
		(|w^{(0)}_{r}|^{2}-|w^{(0)}_{\theta}|^{2})}{|x|^{2}}{\rm d}x \\
	&&- 2 \mu \int_{B_M \setminus \overline{B}} \dfrac{w^{(0)}_rw^{(0)}_{\theta}}{|x|^2} {\rm d}x%
	+\int_{B_M \backslash \overline{B}}\mathbf{w}^{(0)}\cdot \left( \mathbf{w}^{(0)}\cdot
	\nabla \right) \mathbf{v}_{a}{\rm d}x.
\end{eqnarray*}%
Eventually, our identity rewrites:
\begin{align} \notag
& \int_{ B_M \backslash \overline{B}}|\nabla \mathbf{w}^{(0)}|^{2}{\rm d}x +\int_{ B_M 	\backslash \overline{B}}\frac{\phi_0 (|w^{(0)}_{r}|^{2}-|w^{(0)}_{\theta}|^{2})}{ |x|^{2}}{\rm d}x 
  =  2\mu \int_{B_M \backslash \overline{B}} \dfrac{w_r^{(0)}w_{\theta}^{(0)}}{|x|^2} {\rm d}x\\
  & - 
   \int_{ B_M \setminus B} \mathbf{w}^{(0)} \cdot \left[ \left( \mathbf{w}^{(0)}
	\cdot \nabla \right) \mathbf{v}_{a}+ \left( \mathbf{w}_0 \cdot \nabla
	\right) \mathbf{v}_{a}^{(0)}  \right]{\rm d}x
	- \int_{\mathbb B_M \setminus \overline{B}} \mathbf{w}^{(0)} \cdot \left[ \mathbf{v}_b^{(0)}\cdot \nabla \mathbf{w}_0 \right] {\rm d}x
	+ \varepsilon(M) 
\end{align}

Building on this last estimate, {\bf Proposition \ref{prop_largemodes}} yields by combining the following two lemmas and sending then $M$ to infinity.  
 
\begin{lemma} \label{lem_w2}
	Assume $ \phi_0 \in (2,3]$ and $\beta_0 >0.$ Given a divergence-free ${\bf w} \in C^{\infty}(\mathbb R^2 \setminus \overline{B})$ such that:
	\[
	{\bf w}  = 0 \text{ on $\partial B,$} \qquad \sup_{|x| > 1} |x|^{1+\beta_0} |{\bf w}(x)| <  \infty.
	\]
There exists $c > 0$ such that, for all $M > 1,$ there is a constant $\varepsilon(M)$ such that there holds:
	\begin{align} \notag
	\int_{ B_M \backslash \overline{B}}|\nabla \mathbf{w}^{(0)}|^{2}{\rm d}x & +\int_{ B_M \setminus \overline{B}}\frac{\phi_0 (|w^{(0)}_{r}|^{2}-|w^{(0)}_{\theta}|^{2})}{ |x|^{2}}{\rm d}x \\
	&  \geq  \varepsilon(M) +  c\left[  \int_{\mathbb R^2 \setminus \overline{B}} |\nabla \mathbf w^{(1)}|^2{\rm d}x +  \int_{\mathbb R^2 \setminus \overline{B}} \dfrac{|\mathbf w^{(1)}|^2}{|x|^2}{\rm d}x\right].  \label{eq_lemw2}
	\end{align}
	Furthermore $\varepsilon(M) \to 0$ when $M \to \infty.$
\end{lemma}

\begin{lemma} \label{lem_source2}
	Let $\mathbf v$ and $\mathbf w$ be smooth and 
	divergence free in $\mathbb R^2 \setminus \overline{B}$ and assume $\beta_0 >0.$ There holds:
	\[
	\begin{aligned}
		& \left|\int_{B_M \setminus \overline{B}} \mathbf{w}^{(0)} \cdot \left[ \left( \mathbf{w}^{(0)}
	\cdot \nabla \right) \mathbf{v}+ \left( \mathbf{w}_0 \cdot \nabla
 	\right) \mathbf{v}^{(0)}  \right] {\rm d}x \right| \\
 	 &
 	 \leq  \sup_{|x| > 1} (|x|^{2+\beta_0} |\nabla \mathbf v| + |x|^{2+\beta_0} |\nabla \mathbf v^{(0)}|) \int_{B_M \setminus \overline{B} } \dfrac{|\mathbf w|^2 }{|x|^2} {\rm d}x, \\	
\end{aligned}
\]
and
\[
\begin{aligned} 	
		& \left|\int_{B_M \setminus \overline{B} } \mathbf{w}^{(0)} \cdot \left[ \mathbf{v}^{(0)}\cdot \nabla \mathbf{w}_0 \right]{\rm d}x  \right| \\ & \leq  \dfrac{1}{2\eta} \left[\sup_{|x| > 1} (|x|^{1+\beta_0} |{\mathbf v^{(0)}}|)\right]^2 \int_{B_M \setminus \overline{B} } \dfrac{|\mathbf w|^2 }{|x|^2} {\rm d}x + \dfrac{\eta}{2} \int_{B_M \setminus \overline{B}} |\nabla {\bf w}_0|^2{\rm d}x, 
	\end{aligned}
	\]
	whatever $\eta \in (0,1].$
\end{lemma}

We end up this section with proofs for these two lemmas. 
We start with the proof of {\bf Lemma \ref{lem_source2}}  that is straightforward.

\medskip

\begin{proof}[Proof of Lemma \ref{lem_source2}]
	We give a proof of the first inequality only. The second one is similar.  We
	have:
	\begin{eqnarray*}
	&&	\left|\int_{B_M \setminus \overline{B}} \mathbf{w}^{(0)} \cdot \left[ \left( \mathbf{w}^{(0)}
	\cdot \nabla \right) \mathbf{v}+ \left( \mathbf{w}_0 \cdot \nabla
	\right) \mathbf{v}^{(0)}  \right]{\rm d}x\right| \\ &\leq &  \int_{B_M \setminus \overline{B}} |\nabla {\mathbf v}| |\mathbf w^{(0)}|^2{\rm d}x +  \int_{B_M \setminus \overline{B}} { |\nabla {\mathbf v}^{(0)}|}  |\mathbf w^{(0)}||\mathbf w_0| {\rm d}x\\
		& \leq & \int_{B_M  \setminus \overline{B}} |x|^2|\nabla {\mathbf v}| \dfrac{|\mathbf w^{(0)}|^2}{|x|^2} {\rm d}x+  \int_{\mathbb R^2 \setminus \overline{B}} { |x|^2|\nabla {\mathbf v}^{(0)}|} \dfrac{ |\mathbf w^{(0)}|}{|x|}\dfrac{| \mathbf w_0| }{|x|}{\rm d}x.
	\end{eqnarray*}
	We conclude by  standard Hölder inequalities
	noting that, since $|x| > 1$ and $\beta_0 > 0,$ we can bound  $|x|^{2} \leq |x|^{2+\beta_0}$ in $\mathbb R^2 \setminus \overline{B}.$
\end{proof}

\medskip

The proof of {\bf Lemma \ref{lem_w2}} requires more care and is based on fine structure properties of divergence-free vector-fields.

\medskip

\begin{proof}[Proof of Lemma \ref{lem_w2}]
	Let ${\bf w}$ match the assumptions of our lemma. We recall that there is ${\varphi} \in C(\mathbb R^2) \cap C^{\infty}(\mathbb R^2 \setminus \overline{B})$ such that ${\bf w}= \nabla^{\bot} \varphi.$ Since ${\bf w}$ decays strictly faster than $1/r,$ we can normalize this stream function so that it is constant on $B$ and vanishes at infinity. Then $\varphi$ vanishes faster than $1/r^{\beta_0}$ at infinity.  Up to approximating with a truncation argument, we prove \eqref{eq_lemw2} with $\varepsilon(M) = 0$ in the favorable case where ${\bf w}^{(0)} \in C^{\infty}_c(B_M \setminus \overline{B}).$ In this case, all integrals involved in \eqref{eq_lemw2} are actually integrals in $\mathbb R^2 \setminus \overline{B}.$
	
	\medskip
	
	With  this further assumption, we infer on one hand:
	\begin{eqnarray}
\notag		\int_{\mathbb R^2 \backslash \overline{B}}|\nabla \mathbf{w}^{(0)}|^{2}{\rm d}x &=&\int_{1}^{\infty} \int_{\mathbb T} \left| \partial _{r}\mathbf{w}^{(0)}\right| ^{2}+\left| \frac{
\partial _{\theta }\mathbf{w}^{(0)}}{r}\right| ^{2} r{\rm d}r {\rm d}\theta \\
		&=&2
	\notag	\int_{1}^{\infty} \int_{\mathbb T}\left| -\frac{1}{r^{2}}\partial _{\theta}\varphi^{(0)} +\frac{1}{r}\partial _{\theta r}\varphi^{(0)} \right| ^{2} r{\rm d}r{\rm d}\theta \\
\notag		&& +\int_{1}^{\infty} \int_{\mathbb T}\left| \partial _{rr}\varphi^{(0)}\right| ^{2}r{\rm d}r{\rm d}\theta \\
		&&+\int_{1}^{\infty} \int_{\mathbb T} \left| \frac{1}{r^{2}}%
		\partial _{\theta \theta }\varphi^{(0)} +\frac{1}{r}\partial _{r}\varphi^{(0)} \right|^{2}%
		 r{\rm d}r{\rm d}\theta.  \label{eq_nablaw1}
	\end{eqnarray}
In particular, expanding sums and applying that:
\begin{align*}
2\Re \left( \int_{1}^{\infty} \int_{\mathbb T} \dfrac{\partial_{\theta\theta} \varphi^{(0)} \partial_r \bar{\varphi}^{(0)}}{r^2} {\rm d}r{\rm d}\theta\right) & = -2 \Re \left(\int_{1}^{\infty} \int_{\mathbb T} \dfrac{\partial_{\theta} \varphi^{(0)} \partial_{\theta r} \bar{\varphi}^{(0)}}{r^2} {\rm d}r{\rm d}\theta\right) \\
&  = - 2 \int_{1}^{\infty}\int_{\mathbb T} \dfrac{|\partial_{\theta} \varphi^{(0)}|^2}{r^3} {\rm d}r{\rm d}\theta,
\end{align*}		
we have:
\begin{align} \notag  
\int_{\mathbb R^2 \setminus \overline{B}} |\nabla \mathbf{w}^{(0)}|^2{\rm d}x \leq &  2 \int_{1}^{\infty} \int_{\mathbb T}  \frac{|\partial _{\theta r}\varphi^{(0)}|^2}{r^2} r{\rm d}r{\rm d}\theta \\
& + \int_{1}^{\infty} \int_{\mathbb T} \left[  \frac{|\partial _{r}\varphi^{(0)}|^2}{r^2}  + |\partial_{rr}\varphi^{(0)}|^2 + \frac{|\partial _{\theta \theta}\varphi^{(0)}|^2}{r^4} \right] r{\rm d}r{\rm d}\theta.  \label{eq_H1rad} 
\end{align}
This inequality is in the wrong sense to obtain coercivity of our quadratic form.  But,  we provide it right now since the above integration by parts will be very useful  for further computations.
On the other hand, we have: 
	\begin{equation} \label{eq_perturb}
		\int_{\mathbb R^2 \setminus \overline{B}}\frac{(|w^{(0)}_{r}|^{2}-|w^{(0)}_{\theta}|^{2})}{|x|^{2}}%
		{\rm d}x=\int_{1}^{\infty} \int_{\mathbb T}\frac{\left| \frac{1}{r}\partial _{\theta
			}\varphi^{(0)} \right| ^{2}-\left| \partial _{r}\varphi^{(0)} \right| ^{2}}{r^{2}} r{\rm d}r{\rm d} \theta.
	\end{equation}%
Combining \eqref{eq_nablaw1} and \eqref{eq_perturb}, the quantity we want to compute ({\em i.e.} the left-hand side of \eqref{eq_lemw2}) reads:
	\begin{align*}
Q_+  := & 	\int_{\mathbb R^2 \backslash \overline{B}}|\nabla \mathbf{w}^{(0)}|^{2}{\rm d}x + \phi_0 \int_{\mathbb R^2 \backslash \overline{B}}\frac{|w^{(0)}_{r}|^{2}-|w^{(0)}_{\theta}|^{2}}{|x|^{2}}	{\rm d}x\\
		 = & \int_{1}^{\infty} \int_{\mathbb T} \left[2\left| \partial _{r}\left( 
		\frac{1 }{r}\partial _{\theta }\varphi^{(0)} \right) \right| ^{2}+\left| \partial
		_{rr}\varphi^{(0)} \right| ^{2}\right]r{\rm d}r{\rm d}\theta \\
		&+\int_{1}^{\infty} \int_{\mathbb T}\left| \dfrac{\partial _{\theta \theta
			}\varphi^{(0)} }{r^{2}}+\frac{\partial _{r}\varphi^{(0)} }{r}\right| ^{2}r{\rm d}r{\rm d}\theta +\phi_0
		\int_{1}^{\infty} \int_{\mathbb T}\frac{\left| \frac{\partial _{\theta }\varphi^{(0)} }{r%
			}\right| ^{2}-\left| \partial _{r}\varphi^{(0)} \right| ^{2}}{r^{2}} r{\rm d}r{\rm d}\theta.
	\end{align*}
We then split $\varphi$ in Fourier modes and remarking that $Q_+$ does not mix Fourier frequencies, we have also the splitting $Q_+ = Q_1 + Q^{(1)},$ where $Q_1$ and $Q^{(1)}$
stand for $Q_+$ where $\varphi$ is replaced respectively by 1-frequencies and $(1)$-frequencies respectively.

\medskip

Concerning $Q_1,$ we have:
	\begin{align*}
		Q_1 = 2&\sum_{k = \pm 1} \left[\int_1^{\infty}\int_{\mathbb T} \left ( 3 \left| \partial_r \left(\dfrac{\varphi_k}{r}
		\right) \right|^2 + |\partial_{rr} \varphi_k|^2 \right) r \mathrm{d}r{\rm d}\theta \right. \\
		& \left. + \phi_0
		\int_1^{\infty}\int_{\mathbb T} \left( \dfrac{|\varphi_k|^2}{r^4} - \dfrac{|\partial_r \varphi_k |^2 }{r^2%
		} \right) r \mathrm{d}r{\rm d}\theta\right].
	\end{align*}
	We rewrite here the second term. Up to split in real and imaginary parts, we
	can assume that $\varphi_1$ is real and we have then: 
	\begin{equation*}
		\begin{aligned} \int_1^{\infty} \dfrac{|\partial_r \varphi_1|^2}{r^2} r{\rm d} r & =
			\int_{1}^{\infty} \left( \left| \partial_r \left(\dfrac{\varphi_1}{r} \right)
			\right|^2 + \dfrac{|\varphi_1|^2}{r^4} + 2 \partial_r \left(\dfrac{\varphi_1}{r}
			\right) \dfrac{\varphi_1}{r^2} \right) r{\rm d}r \\ & = \int_{1}^{\infty} \left(
			\left| \partial_r \left(\dfrac{\varphi_1}{r} \right) \right|^2 +
			\dfrac{|\varphi_1|^2}{r^4}\right) r{\rm d}r. \end{aligned}
	\end{equation*}
	With similar computations in the case $k=-1,$ we obtain: 
	\begin{equation*}
		Q_1 = 2 { \sum_{k=\pm 1}}\int_1^{\infty}\int_{\mathbb T} \left( (3-\phi_0) \left| \partial_r \left(\dfrac{\varphi_k}{r}
		\right) \right|^2 + |\partial_{rr} \varphi_k|^2 \right) r \mathrm{d}r{\rm d}\theta \geq 0.
	\end{equation*}

As for $Q^{(1)},$ by remarking that
\begin{eqnarray*}
		\int_{1}^{\infty} \int_{\mathbb T}\left\vert \frac{\partial _{r}\varphi^{(1)} }{r}%
		\right\vert ^{2}r{\rm d}r{\rm d}\theta
	&=&\int_{1}^{\infty} \int_{\mathbb T} \left|\frac{\partial _{\theta \theta }\varphi^{(1)}}{r^{2}} +\frac{\partial
			_{r}\varphi^{(1)}}{r} \right|^2r{\rm d}r{\rm d}\theta \\
		&&-\int_{1}^{\infty} \int_{\mathbb T}\left| \frac{\partial_{\theta\theta }\varphi^{(1)} }{r^{2}}\right| ^{2}r{\rm d}r{\rm d}\theta   +2\int_{1}^{\infty} \int_{\mathbb T}\left| \frac{\partial _{\theta }\varphi^{(1)} }{r^{2}}\right| ^{2}r{\rm d}r{\rm d} \theta,
	\end{eqnarray*}
we expand:
	\begin{eqnarray*}
		Q^{(1)} &=&\left( 1-\phi_0 \right) \int_{1}^{\infty} \int_{\mathbb T}\left| \frac{1}{%
			r^{2}}\partial _{\theta \theta }\varphi^{(1)} +\frac{1}{r}\partial
		_{r}\varphi^{(1)} \right| ^{2}r{\rm d}r{\rm d}\theta \\
		&& +\int_{1}^{\infty} \left(\int_{\mathbb T}2\left| -\frac{\partial _{\theta}\varphi^{(1)}}{r^{2}} +\frac{\partial _{\theta r}\varphi^{(1)}}{r} \right| ^{2}+\left|
		\partial _{rr}\varphi^{(1)} \right| ^{2} \right)r{\rm d}r{\rm d}\theta \\
		&&-\phi_0 \int_{1}^{\infty} \int_{\mathbb T}\frac{\left| \partial _{\theta}\varphi^{(1)} \right| ^{2}}{r^{4}}r{\rm d}r{\rm d}\theta+ \phi_0\int_{1}^{\infty} \int_{\mathbb T}\left| \frac{\partial _{\theta \theta }\varphi^{(1)} }{r^{2}}\right| ^{2}r{\rm d}r{\rm d}\theta. \\
		\end{eqnarray*}
We expand now all squared-norms to obtain that
\begin{eqnarray*}		
	Q^{(1)}	&=& \int_{1}^{\infty} \int_{\mathbb T}\left| \frac{1}{r^{2}}\partial _{\theta\theta }\varphi^{(1)}\right| ^{2} r{\rm d}r{\rm d}\theta + (1-\phi_0) \int_{1}^{\infty} \int_{\mathbb T} \left|\frac{1}{r}\partial _{r}\varphi^{(1)} \right| ^{2}r{\rm d}r{\rm d} \theta \\
		&& + (\phi_0 - 4) \int_{1}^{\infty} \int_{\mathbb T}\left| \frac{\partial _{\theta}\varphi^{(1)}}{r^{2}}\right|^2
		r{\rm d}r{\rm d}\theta  +2\int_{1}^{\infty} \int_{\mathbb T} \left| \frac{\partial _{\theta r}\varphi^{(1)}}{r} \right|^{2}
		r{\rm d}r{\rm d}\theta \\
		&& +\int_{1}^{\infty} \int_{\mathbb T}\left| \partial _{rr}\varphi^{(1)} \right| ^{2} r{\rm d}r{\rm d}\theta.
	\end{eqnarray*}%
	
	At this point, we want to take advantage of the Poincar\'{e} Wirtinger inequality that
	states (in this case where we consider only modes that are larger than $2$): 
	\begin{equation*}
		\int_{\mathbb T} |\partial_{\theta \theta} \varphi^{(1)}|^2 \mathrm{d}\theta \geq
		4 \int_{\mathbb T} |\partial_{\theta} \varphi^{(1)}|^2 \mathrm{d}\theta, \qquad
		\int_{\mathbb T} |\partial_{\theta r} \varphi^{(1)}|^2 \mathrm{d}\theta \geq 4
		\int_{\mathbb T} |\partial_{r} \varphi^{(1)}|^2 \mathrm{d}\theta.
	\end{equation*}
	For arbitrary $\phi_0 \geq 0,$ we get: 
	\begin{equation*}
		\begin{aligned} Q^{(1)} \geq & \phi_0 \left[ \int_{1}^{\infty} \int_{\mathbb T}
			\dfrac{1}{r^4}|\partial_{\theta} \varphi^{(1)} |^2 r{\rm d}r{\rm d}\theta\right] + \left( 2 -
			\dfrac{\phi_0}{4} \right) \int_{1}^{\infty} \int_{\mathbb T}\left| \frac{\partial
				_{\theta r}\varphi^{(1)}}{r}\right|^2 r{\rm d}r{\rm d}\theta \\ & \; + \int_{1}^{\infty} \int_{\mathbb T} \left| \partial _{r r}\varphi^{(1)} \right|^{2}r{\rm d}r{\rm d}\theta +\int_{1}^{\infty} \int_{\mathbb T} \left|\frac{\partial _{r}\varphi^{(1)}}{r} \right|^{2}r{\rm d}r{\rm d}\theta.
		\end{aligned}
	\end{equation*}
	But whenever $\phi_0 \in (2,3]$, we have $2-\phi_0/4 \geq 5/4 $. If we split $\phi_0-4=(\phi_0-1)-3$, then one has by the above Poincar\'{e} Wirtinger inequality that: 
	\begin{equation*}
		\begin{aligned} Q^{(1)} \geq & \dfrac{1}{4} \int_{1}^{\infty} \int_{\mathbb T} \dfrac{1}{r^4}|\partial_{\theta \theta} \varphi^{(1)} |^2r{\rm d}r{\rm d}\theta + \dfrac{1}{4} \int_{1}^{\infty} \int_{\mathbb T} \dfrac{1}{r^4}|\partial_{\theta } \varphi^{(1)} |^2r{\rm d}r{\rm d}\theta
			 \\ & + \int_{1}^{\infty} \int_{\mathbb T} \left| \frac{\partial_{\theta r}\varphi^{(1)}}{r}\right|^2r{\rm d}r{\rm d}\theta + 2\int_{1}^{\infty} \int_{\mathbb T}\left|\frac{\partial _{r}\varphi^{(1)}}{r} \right|^{2} r{\rm d}r{\rm d}\theta\\ & \; + \int_{1}^{\infty} \int_{\mathbb T} \left| \partial _{rr}\varphi^{(1)} \right|^{2}r{\rm d}r{\rm d}\theta.
		\end{aligned}
	\end{equation*}
	
We conclude by applying the bound \eqref{eq_H1rad} to compute the $L^2$ norm of $\nabla w^{(1)}$ together with:
\[
\int_{\mathbb R^2 \setminus \overline{B}} \dfrac{|{\bf w}^{(1)}|^2}{|x|^2}{\rm d}x = \int_{1}^{\infty} \int_{\mathbb T}\left(  \dfrac{|\partial_r \varphi^{(1)}|^2 }{r^2} + \dfrac{|\partial_{\theta} \varphi^{(1)}|^2}{r^4} \right) r{\rm d}r{\rm d}\theta.
\]
\end{proof}



\appendix

\section{Proof of Lemma \ref{lem_analysislinearproblem}}
\label{app_linearized}
We recall that
\begin{equation*}
\zeta _{n}^{\pm }=-\frac{\phi_0 }{2}\pm \frac{1}{2}\sqrt{\phi_0 ^{2}+4(in\mu
+n^{2})},\text{ \ }\forall n\in \mathbb{Z}\text{.}
\end{equation*}%
{We can compute that the real part of $\zeta _{n}^{\pm }$ is given by 
\begin{eqnarray}
{ \xi_n^{+}}  = 
\Re (\zeta _{n}^{+}) &=&-\frac{\phi_0 }{2}+\left( \frac{1}{\sqrt{2}}\right)
^{3}\left[ (\phi_0 ^{2}+4n^{2})+\sqrt{(\phi_0 ^{2}+4n^{2})^{2}+16n^{2}\mu^{2}}%
\right] ^{1/2}>0,  \label{real part 1} \\
{ \xi_n^{-}}  =  \Re (\zeta _{n}^{-}) &=&-\frac{\phi_0 }{2}-\left( \frac{1}{\sqrt{2}}\right)
^{3}\left[ (\phi_0 ^{2}+4n^{2})+\sqrt{(\phi_0 ^{2}+4n^{2})^{2}+16n^{2}\mu^{2}}%
\right] ^{1/2}<0.  \label{real part 2}
\end{eqnarray}%
}In what follows, we use without mention the following properties of $\zeta
_{n}^{\pm}:$%
\begin{proposition} \label{prop_zetancontrol}
	Let $\phi_0 \geq 0$ and $I$ a compact interval of $\mathbb R.$ The following statements hold true for all $n\in \mathbb Z:$
	\begin{enumerate}[i.]
		\item the mapping $\mu \mapsto (\zeta_n^+,\zeta_n^-)$ is smooth on $I$ 
		\item  there exist constants $0 < C_m < C_M$ depending only on $\phi_0$ and $I$ for which whatever $\mu \in I$ we have:
		\begin{align*}
			C_m (1+|n|) & \leq |{ \xi_n^+}| \leq |\zeta_n^+| \leq C_M(1+|n|), \\
			C_m (1+|n|) & \leq |{ \xi_n^-}| \leq |\zeta_n^-| \leq C_M(1+|n|), \\
			C_m(1+|n|)& \leq \left |\sqrt{\phi_0^2 + 4(in\mu + n^2)} \right |. 
		\end{align*}
	\end{enumerate}
\end{proposition}
The proof is purely technical and left to the reader.
At this point, we remark that $|{\xi_{n}^-}|$ is an increasing function of $|n|$ so that its minimal value (over $n\in \mathbb{Z}\backslash
\{0\}$) is reached for $n=\pm 1$ and is equal to%
\begin{equation*}
\rho _{\mu ,\phi_0 }=\frac{\phi_0 }{2}+\frac{1}{2\sqrt{2}}\left( \phi_0 ^{2}+4+%
\sqrt{\left( \phi_0 ^{2}+4\right) ^{2}+16\mu ^{2}}\right) ^{1/2}.
\end{equation*}%
With assumption \eqref{eq_conditionphimu} we have $\rho_{\mu,\phi_0} > 2.$
For convenience later on, we introduce then:
\[
\alpha_{\mu,\phi_0}=\frac{1}{2} \min (\rho_{\mu,\phi_0}-2,1).
\]
We remark now that the formula \eqref{eq_wn}-\eqref{eq_gamman} defining $(\hat{\gamma},\hat{w})$
splits into the difference between $(\hat{\gamma}^{(b)},\hat{w}^{(b)})$ and $(\hat{\gamma}[\hat{F}],\hat{w}[\hat{F}])$ with the obvious convention that:
\[
(\hat{w}[\hat{F}])_n = w_n[F_n], \qquad (\hat{\gamma}[\hat{F}])_n = \gamma_n[F_n] ,  \quad  \quad \forall \, n \in \mathbb Z,
\]
as given by the formulas \eqref{eq_wnF} and \eqref{eq_gammanF} for $n \neq 0$ and \eqref{eq_w0<2} and \eqref{eq_gamma0<2} in case $n=0.$ The terms with exponent $(b)$ are due to the reflection of these solutions on the  boundary $r=1.$ When $n\neq 0,$ it reads:
\begin{equation}
\begin{aligned}
&{w}^{(b)}_n= \bar{w}_n r^{\zeta_n^-}, \\
&{\gamma}^{(b)}_n = 
\left\{
\begin{aligned}
& \frac{\bar{\gamma}_{n}}{r^{|n|}} 
-\frac{\bar{w}_n}{(\zeta_n^{-} +2)^2 - n^2} r^{2+\zeta_n^-},
 && \text{if } (\zeta_n^{-}+ 2)^2 \neq n^2\\
&  \frac{\bar{\gamma}_{n}}{r^{|n|}} 
+ \frac{\bar{w}_n}{2|n| r^{|n|}}\ln r,
&& \text{if } \zeta_n^{-}+ 2 + |n| = 0
\end{aligned}
\right.
\end{aligned}
\label{eq_boundarytermn}
\end{equation}
with $(\bar{\gamma}_n,\bar{w}_n)$ given by \eqref{eq_constantgammaw}-\eqref{eq_constantgammawresonant-new}. 
In case $n=0$ we have ${w}^{(b)}_0(r) = {\gamma}^{(b)}_0(r) = 0$ if $\phi \leq 2$
 while for  $\phi > 2$:
 \begin{equation} \label{eq_boundaryterm0}
{w}_0^{(b)}(r) =  \dfrac{\bar{w}_0}{r^{\phi_0}}, \qquad {\gamma}_0^{(b)}(r) = \dfrac{\bar{w}_0}{(\phi_0-2)^2 r^{\phi_0-2}},
 \end{equation}
with $\bar{w}_0$ given by \eqref{eq_constantgamma0}.
We prove now successively that both  $(\hat{\gamma}[\hat{F}],\hat{w}[\hat{F}])$ and $(\hat{\gamma}^{(b)},\hat{w}^{(b)})$ satisfy the conclusion of {\bf Lemma \ref{lem_analysislinearproblem}}. This shall end the proof.

\paragraph{\em The case of bulk terms.} Fix $\phi_0 \in [0,\infty)$  and $\mu \in \mathbb R$ such that \eqref{eq_conditionphimu} holds. Then, fix $\alpha < \alpha_{\mu,\phi_0}$ and $\hat{F} \in \mathcal B_{4+2\alpha,\kappa+1}.$ Plugging the results of Proposition \ref{prop_zetancontrol} into the computations of $\gamma_n[F_n],w_n[F_n]$ yields, as in the proof of \cite[Proposition 11]{HillairetWittwer}, that 
\begin{multline*}
\sup_{r > 1} \left(  (1+|n|)^{\kappa+3} r^{\alpha+2} |w_n[F_n](r)|   +  (1+|n|)^{\kappa+2} r^{\alpha+3} |\partial_r w_n[F_n](r)|\right. \\
\left.+  (1+|n|)^{\kappa+1} r^{\alpha+4} |\partial_{rr} w_n[F_n](r)| \right)
 \leq C_{M} \|\hat{F} ; \mathcal B_{4+2\alpha,\kappa+1} \| 
\end{multline*}
possibly with a larger $C_M.$ 
In case $n=0,$ we have with similar computations:
\begin{multline*}
\sup_{r > 1} \left( r^{\alpha+2} |w_0[F_0](r)|   + r^{\alpha+3} |\partial_r w_0[F_0](r)| +  r^{\alpha+4} |\partial_{rr} w_0[F_0](r)| \right)
 \leq C \sup_{r >1}r^{4+2\alpha}|F_0(r)|.
\end{multline*}
We conclude that $\hat{w}[\hat{F}] \in \mathcal U_{\alpha+2,\kappa+3}$  with:
\[
\|\hat{w}[\hat{F}] ; \mathcal U_{\alpha+2,\kappa+3} \| \leq C_{M} \|\hat{F} ; \mathcal B_{4+2\alpha,\kappa+1} \| . 
\]

We can prove then that $\hat{\gamma}[\hat{F}] \in \mathcal U_{\alpha,\kappa+5}$ by applying Proposition 12 of \cite{HillairetWittwer} with $\kappa$ replaced by $\kappa+1$ and in case $\hat{\phi} = \hat{w}[\hat{F}]$ and $\hat{\gamma}^* =0.$ This entails also by considering the above control on $\hat{w}[\hat{F}]$:
\[
\|\hat{\gamma}[\hat{F}] ; \mathcal U_{\alpha,\kappa+5} \| \leq C_{M} \|\hat{F} ; \mathcal B_{4+2\alpha,\kappa+1} \| . 
\] 

When $\mu$ varies in an interval $I$, we note that all quantities involved in the computation of $\gamma_n[F_n],w_n[F_n]$ (but $F_n$) depend smoothly on $\mu.$ Choosing $\alpha < \min \{\alpha_{\mu,\phi_0}, \mu  \in I \},$ standard parameter-integral arguments entail that the mapping $\mu \to (\hat{F} \mapsto \hat{w}[\hat{F}])$ is continuous from $I$ 
into $\mathcal L_c(\mathcal B_{4+2\alpha,\kappa+1}; \mathcal U^2_{\alpha,\kappa+4}\times \mathcal U^2_{\alpha+2,\kappa+3}).$ We refer to \cite[Section 5.2]{HillairetWittwer} for similar computations. 

\paragraph{\em The case of boundary terms.}
We analyze now the formulas \eqref{eq_boundarytermn}
for $n\neq 0$ with $(\bar{w}_n,\bar{\gamma}_n)$ given by \eqref{eq_constantgammaw}-\eqref{eq_constantgammawresonant-new}, and \eqref{eq_boundaryterm0} in case $n=0$ and $\phi_0 > 2.$
We consider $\alpha < \alpha_{\mu,\phi_0},$ and $\hat{F} \in \mathcal B_{4+2\alpha,\kappa+1}.$ 
By construction of $\alpha > 0$ since $|\Re(\zeta_{n}^-)|>\alpha+2$, we have that:
\[
\sup_{r > 1} \left( r^{\alpha+2} |r^{\zeta_n^-}|   + \dfrac{r^{\alpha+3}}{(1+|n|)} |\partial_{r}r^ {\zeta_n^-}| + \dfrac{r^{\alpha+4}}{(1+|n|)^2} |\partial_{rr}r^{\zeta_n^-}| \right)
 \leq C_0,
\]
and similarly:
\begin{multline*}
\sup_{r > 1} \left( r^{\alpha} (|r^{\zeta_n^- +2}| +r^{-|n|})   + \dfrac{r^{\alpha+1}}{(1+|n|)} (|\partial_{r}r^ {\zeta_n^- +2}| + |\partial_r r^{-|n|}|) \right. \\ \left.
+ \dfrac{r^{\alpha+2}}{(1+|n|)^2} ( |\partial_{rr}r^{\zeta_n^-+2}|  + |\partial_{rr} r^{-|n|}|) \right)
 \leq C_0,
\end{multline*}
with a constant $C_0 $ that depends on $\phi_0$ and $\mu$ but that remains bounded in $\mu$ ranging a compact interval of $\mathbb R.$

To control $(\gamma_n^{(b)},w_n^{(b)})_{n\neq 0}$, we only need to get information on $(\bar{\gamma}_n,\bar{w}_n)_{n \neq 0}.$ For this, we first remark 
that, by standard trace arguments, there holds
\[
(\gamma_n[F_n](1),\partial_r \gamma_n[F_n](1))_{n\neq 0} \in   \mathcal{B}^{0}_{\kappa+5} \times \mathcal B^{0}_{\kappa+4},
\]
with
\[
\|(\gamma_n[F_n](1),\partial_r \gamma_n[F_n(1))_{n\neq 0} ;\mathcal{B}^{0}_{\kappa+5} \times \mathcal B^{0}_{\kappa+4} \| \leq C_M \|\hat{F} ; {\mathcal B_{4+2\alpha,\kappa+1}}\|.
\]
We note that, in the latter estimates, we implictly complemented the sequence defined for values $n \neq 0$ by $0$ in case $n=0.$

\medskip

At this point, we realize that, for large values of $n$ we have:
\[
2 + \zeta_n^{-} = \left( 2- \frac{\phi_0}{2} \right)  -  |n| \sqrt{ 1   +  \frac{i\mu}{|n|} + \frac{\phi_0^2}{4n^2}  } =  -|n| + \left( 2 - \frac{\phi_0}{2} - \dfrac{i\mu}{2} \right)+ O(1/|n|). 
\]
In particular $(2+\zeta_n^-) + |n|$ can vanish only for  $|n| < N_0(\phi_0).$  For larger values, we have:
\[
|2 + \zeta_n^{-} | \leq 2 |n|, \qquad |2 + \zeta_n^- + |n| | \geq C(\phi_0,I).
\]

Consequently, for $|n| \geq N_0(\phi_0,I)$ we cannot have $2+\zeta_n^-+ |n| = 0$ and plugging the above observations in the formulas \eqref{eq_constantgammaw} we get:
\[
|\bar{\gamma}_n| \leq C(\phi_0,I) \left( |n| (|\gamma_n[F_n](1)| + |v_{r,n}^*|) + |\partial_r \gamma_n[F_n](1)| + |v_{\theta,n}^*|\right)
\] 
and 
\[
|\bar{w}_n| \leq C(\phi_0,I) |n| \left( |n| (|\gamma_n[F_n](1)| + |v_{r,n}^*|) + |\partial_r \gamma_n[F_n](1)| + |v_{\theta,n}^*|
\right).
\]
By a direct inspection of formulas \eqref{eq_constantgammaw}-\eqref{eq_constantgammawresonant-new}-\eqref{eq_constantgamma0} for $|n| \leq N_0(\phi_0,I)$ we obtain finally that
\[
(\bar{\gamma}_n,\bar{w}_n)_{n \in \mathbb Z} \in   \mathcal{B}^{0}_{\kappa+4} \times \mathcal B^{0}_{\kappa+3},
\]
with
\[
\|(\bar{\gamma}_n,\bar{w}_n)_{n \in \mathbb Z};\mathcal{B}^{0}_{\kappa+4} \times \mathcal B^{0}_{\kappa+3} \| \leq C_M \left( \|\hat{F} ; {\mathcal B_{4+2\alpha,\kappa+1}}\| + \|\hat{v}^*_{r} ; \mathcal B^{0}_{\kappa+5}\| + \|\hat{v}^*_{\theta} ; \mathcal B^{0}_{\kappa+4}\|\right).
\]
We infer that
\[
(\hat{\gamma}^{(b)},\hat{w}^{(b)}) \in  \mathcal{U}^{2}_{\alpha,\kappa+4} \times \mathcal U^{2}_{\alpha+2,\kappa+3},
\]
with
\[
\|(\hat{\gamma}^{(b)},\hat{w}^{(b)});\mathcal{U}^{2}_{\alpha,\kappa+4} \times \mathcal U^{2}_{\alpha+2,\kappa+3} \| \leq C_M \left( \|\hat{F} ; {\mathcal B_{4+2\alpha,\kappa+1}}\| + \|\hat{v}^*_{r} ; \mathcal B^{0}_{\kappa+5}\| + \|\hat{v}^*_{\theta} ; \mathcal B^{0}_{\kappa+4}\|\right).
\]
The smoothness of the $\mu$ dependencies for the mapping $\hat{F} \mapsto  (\hat{\gamma}^{(b)},\hat{w}^{(b)})$ is a direct consequence to the smoothness of the mapping $\mu \to (\zeta_n^+,\zeta_n^-)$ with the observation that the decay of $(\hat{\gamma}^{(b)},\hat{w}^{(b)})$ is uniformly controlled on a bounded segment.



\section{Proof of Proposition \ref{prop_decaymode1}}

\label{app_decaymode1}

For simplicity, we drop the $b$ index that has no influence. 
We assume that $({\bf u},p)$ is a smooth solution to \eqref{eq_NS2D}. Like in {\bf Section \ref{sec_existence}},
we expand ${\bf u} = {\bf u}_{ref}[\phi_0,\mu] + {\bf v}$ with
${\bf v} =  \nabla^{\bot} \gamma$ and $\nabla \times {\bf u} = -\Delta \gamma = w,$ recall that $(\gamma_n, w_n)_{n\in\mathbb Z}$ satisfy (\ref{odes}) and match the boundary/asymptotic conditions (\ref{bound gamma n}).
We point out that, because of assumption \eqref{eq_Miref}, one has
\begin{equation} \label{eq_decayboundary}
 \sum_{n \in \mathbb Z} \left[(1+|n|^2)  (|v_{r,n}^*| + |v_{\theta,n}^*|) \right]^2 \leq \|{\bf v}\|^2_{C^2(\overline{B(0,2)} \setminus B)} \leq  |M^{ref}|^2.
\end{equation}

We start by a lemma computing the decay of $(\gamma_n,w_n)$ (for $n$ fixed) knowing a priori decay of the data $F_n:$
\begin{lemma} \label{lem_decaynmode}
Let $\phi_0 \in (21/10,3]$, given $\varepsilon \in (0,1)$ and $\alpha \in (0,3),$ there exist constants $\beta_1 \in (1/2,1)$ and $C := C(\varepsilon,\alpha)$ such that, if $F_n$ satisfies:
\[
|F_n(r)| \leq N_{\alpha}^2 r^{-(4+2\alpha)},  \qquad \forall \, r \geq 1,
\] 
for $N_{\alpha} \geq 0,$ then we have: 
 \begin{itemize}
 \item[(i)]for $n=0$,
\[
|\gamma_0(r)| + r |\partial_r \gamma_0(r)| + r^2 |w_0(r)| + r^3|\partial_r w_0(r)| \leq  (|\mu_0 - \mu| + C N_{\alpha}^2   ) r^{2-\phi_0} + C N_{\alpha}^2 r^{-2\alpha}. 
\]
\item[(ii)] for $|n|=1$,
\begin{align}
 |w_n(r)|  + {r|\partial_r w_n(r)|} \leq  &  C \left(  |v_{r,n}^*| +  |v_{\theta,_n}^*| + {N_{\alpha}^2} \right)  r^{\max(-(2+\beta_1) ,- (2+(2-\varepsilon)\alpha))},  \label{eq_resultw1}\\
 |\gamma_n(r)|  + {r|\partial_r \gamma_n(r)|}  \leq &  C \left(  |v_{r,n}^*| +  |v_{\theta,_n}^*| + {N_{\alpha}^2} \right)  r^{\max(-\beta_1,-(2-\varepsilon)\alpha)}. \label{eq_resultgamma1}
\end{align}
 \item[(iii)] for $|n| \geq 2$,
 \end{itemize}
\begin{align}
 |w_n(r)|  + \dfrac{r|\partial_r w_n(r)|}{(1+|n|)}  \leq  &  C (1+|n|)\left(  |v_{r,n}^*| +  |v_{\theta,_n}^*| + \dfrac{N_{\alpha}^2}{(1+|n|)^3} \right) r^{\max(\xi_2^- ,- (2+(2-\varepsilon)\alpha))},  \label{eq_resultwn}\\
 |\gamma_n(r)|  + \dfrac{r|\partial_r \gamma_n(r)|}{(1+|n|)} \leq & C \left(  |v_{r,n}^*| +  |v_{\theta,_n}^*| + \dfrac{N_{\alpha}^2}{(1+|n|)^3} \right) r^{\max(2+\xi_2^-,-(2-\varepsilon)\alpha,-|n|)}. \label{eq_resultgamman}
\end{align}
\end{lemma}
\begin{proof}
In the proof, all constants $C$ may depend on $\alpha$ and $\varepsilon.$ But they are independent of $n$ and the data. Following the computations in the existence part, we know that, given $n \in \mathbb Z,$ there is only one solution $(w_n,\gamma_n)$ to (\ref{odes}) and (\ref{bound gamma n}).

In the case $n\neq 0$, we have the formula \eqref{eq_wn} for $w_n:$ 
\[
w_n = \bar{w}_n r^{\zeta_n^-} - w_n[F_n](r),
\]
where:
\[
w_n[F_n](r) = \dfrac{1}{\sqrt{\phi_0^2 + 4(in\mu + n^2)}}  \left( \int_{r}^{\infty} s F_n(s) \left(\dfrac{r}{s}\right)^{\zeta_n^+}{\rm d}s +  \int_{1}^{r}  s F_n(s) \left(\dfrac{r}{s}\right)^{\zeta_n^-}{\rm d}s \right).
\]
Plugging our assumption on the decay of $F_n$ with the remark that $\xi_n^+ ={\Re}(\zeta_n^+) \geq C(1+|n|) > 0$ (see {\bf Proposition \ref{prop_zetancontrol}}), we have:
\[
\begin{aligned}
\left|\int_{r}^{\infty} s F_n(s) \left(\dfrac{r}{s}\right)^{\zeta_n^+} {\rm d}s\right|
& \leq r^{\xi_n^+} N_{\alpha}^2 \int_{r}^{\infty} s^{1 -(\xi_n^+  + 4 + 2\alpha) } ds
\leq \dfrac{N_{\alpha}^2}{|\xi_n^+  + 2 + 2\alpha|} r^{-2-2\alpha}\\
& \leq \dfrac{C}{1+ |n|}  \dfrac{N_{\alpha}^2}{r^{2+2\alpha}}. 
\end{aligned}
\]
We proceed with similar computation for the other terms. We note here that $\xi_n^{-} = {\Re}(\zeta_{n}^-) < 0$, so that we might have $\xi_n^{-} + 2 + 2\alpha = 0.$ In this case we should see a log appearing in the integral. We handle this log term by allowing a small loss in the control on the growth of the integral (thus the power $(2-\varepsilon/2)\alpha$ instead of $2\alpha$). The appearing constant $C$ then depends on $\alpha$ and $\varepsilon$. Furthermore, since $\alpha < 3$ and $\xi_{n}^{-} \geq C(1+|n|)$ (see again {\bf Proposition \ref{prop_zetancontrol}}) this might happen only for $|n|$ in a bounded set. We have then the following bounds:
\[
\begin{aligned}
\left|\int_{1}^{r} s F_n(s) \left(\dfrac{r}{s}\right)^{\zeta_n^-} {\rm d}s\right|
& \leq r^{\xi_n^-} N_{\alpha}^2 \int_{1}^{r} s^{1 -(\xi_n^-  + 4 + 2\alpha) }{\rm d}s \\
& \leq \dfrac{C N_{\alpha}^2}{\max(1,|\xi_n^-  + 2 + 2\alpha|)} \left( r^{-2-(2-\varepsilon/2)\alpha} + r^{\xi_n^-} \right)\\
& \leq \dfrac{C N_{\alpha}^2 }{1+ |n|}
\left(\dfrac{1}{r^{2+(2-\varepsilon/2)\alpha}} + r^{\xi_n^-} \right) .
\end{aligned}
\]
Since $\left|\sqrt{\phi_0^2 + 4(in\mu + n^2)}\right| \geq \sqrt{\phi_0^2 + 4n^2} \geq C(1+|n|)$, we conclude that:
\[
|w_n[F_n](r)|  + \dfrac{r |\partial_r w_n[F_n](r)|}{1+|n|}  \leq  \dfrac{CN_{\alpha}^2}{(1+|n|)^2}  r^{ \max(\xi_n^-,-[2+(2-\varepsilon/2)\alpha])} 
\]
that entails \eqref{eq_resultwn}.

We proceed with $\gamma_n.$ Again, with our results in the existence part, we know that there is a coefficient $\bar{\gamma}_{n}$ for which

\[
\gamma_n(r) =  \frac{\bar{\gamma}_{n}}{r^{|n|}} 
- \bar{w}_n \gamma_n^{(H)}(r)
- \gamma_n[F_n](r) ,
\]
where 
\[
\gamma_n[F_n](r) = \int_{r}^{\infty }\frac{sw_{n}[F_n](s)}{2|n|}\left( \frac{r}{s}\right) ^{|n|}{\rm d}s+\int_{1}^{r}\frac{sw_{n}[F_n](s)}{2|n|}\left( \frac{s}{r}\right)^{|n|}{\rm d}s,
\]
and 
\[
\gamma_n^{(H)}(r) = \left\{
\begin{aligned}
\dfrac{1}{(\zeta_n^- +2)^2 - n^2} r^{\zeta_n^- + 2}, && \text{ if $(\zeta_n^- +2)^2 - n^2 \neq 0$}, \\
\dfrac{1}{2|n|} \ln(r) r^{-|n|},&& \text{ if $(\zeta_n^- +2)^2 - n^2 = 0$}.
\end{aligned}
\right.
\]

Plugging the above decay for $w_n[F_n]$, we obtain with the same computations as above:
\[
\left| \int_{r}^{\infty }sw_{n}[F_n](s) \left( \frac{r}{s}\right) ^{|n|}ds \right| \leq \dfrac{CN_{\alpha}^2}{(1+|n|)^3} r^{\max(2+\xi_n^-,-(2-\varepsilon/2)\alpha/4)} ,  
\]
where, since $\phi_0 > 21/10$ we have $\xi_n^{-} < -2-\beta_1.$
For the last term, we use the same remark and note further that $|\xi_n^{-}|$ is increasing with $|n|,$ we obtain:
\[
\int_{1}^{r}sw_{n}[F_n](s) \left( \frac{s}{r}\right)^{|n|}ds \leq  \dfrac{CN_{\alpha}^2}{(1+|n|)^3} \left( \dfrac{1}{r^{|n|}}+ r^{\max(e(n),-(2-\varepsilon)\alpha)} \right)
\]
with $e(n) =  -\beta_1 $ if $|n|=1,$ or
$e(n) = (2+\xi_2^-)$ if $|n| > 2$. Consequently:
\[
|\gamma_n[F_n](r)|  + \dfrac{r |\partial_r \gamma_n[F_n](r)|}{1+|n|}  \leq \dfrac{CN_{\alpha}^2}{(1+|n|)^4} r^{\max(e(n),-|n|,-(2-\varepsilon)\alpha)}.
\]
We point out that we could improve $e(n)$ by using that $w_n[F_n]$ decays like $\xi_n^-$ further than $\xi_{2}^-.$ But for large $|n|$ we would then loose one power of $|n|$ in the pre-factor that we will miss in the next computations. 

We proceed similarly to estimate $\gamma_n^{(H)}$ recalling that $|(\zeta_n^- +2)^2 - |n|^2 | \geq C(1+|n|)$ for $|n| \geq 2.$ We eventually conclude that:
\[
| \gamma_n^{(H)}(r)| \leq \dfrac{C}{(1+|n|)} r^{-|n|+ \delta_1(n)(1-\beta_1)}, 
\]
where $\delta_1(n)=1$ if $|n|=1,$ or
$\delta_1(n) = 0$ if $|n| > 2$,
and 
\[
|\gamma_n(r)|  + \dfrac{r |\partial_r \gamma_n(r)|}{1+|n|}  \leq C\left(  \dfrac{|\bar{w}_n|}{ (1+|n|)} + \dfrac{N_{\alpha}^2}{(1+|n|)^4} \right) r^{\max(e(n),-|n|,-(2-\varepsilon)\alpha)}.
\]
To conclude, we recall that $\bar{w}_n,\bar{\gamma}_n$ are obtained by solving \eqref{eq_constantgammaw}-\eqref{eq_constantgammawresonant-new}:

\[
|\bar{\gamma}_n| +  \dfrac{|\bar{w}_n|}{1+|n|} \leq  C \left(  |v_{r,n}^*| +  |v_{\theta,_n}^*| + \dfrac{N_{\alpha}^2}{(1+|n|)^3}\right) .
\]

The case $n=0$ yields as a direct application of formulas
\eqref{eq_w0>2}-\eqref{eq_gamma0>2}-\eqref{eq_constantgamma0}.
\end{proof}

\medskip

To bootstrap the information on the decay at infinity of ${\bf v}$, we will use the a priori decay of ${\bf v}$ given by {\bf Proposition \ref{prop_decaymode1}} and show via solving (\ref{odes})-(\ref{bound gamma n}) that we have actually a better decay. To this end, we must take advantage of the splitting in Fourier series of ${\bf v},w$ when computing the $F_n.$  This is the content of the following lemma:

\begin{lemma} \label{lem_decayconvol}
Assume that there are constants $C_0>0$, $M_*^{ref}, $ $\beta_0^* >0$ and $\beta_*^{(0)}>0$ such that:
\[
\begin{aligned}
|\gamma_n(r)| + r\dfrac{|\partial_r \gamma_n(r)|}{1+|n|} & \leq   
\dfrac{C_0}{r^{\beta^{(0)}_*}} \left( |v_{r,n}^*| + |v_{\theta,n}^*| +  \dfrac{|M_{*}^{ref}|^2}{(1+|n|)^3} \right), \\
\dfrac{|w_n(r)|}{1+|n|} +  r\dfrac{|\partial_r w_n(r)|}{(1+|n|)^2} &
\leq 
\dfrac{C_0}{r^{2+ \beta^{(0)}_*}} \left( |v_{r,n}^*| + |v_{\theta,n}^*| +  \dfrac{|M_{*}^{ref}|^2}{(1+|n|)^3} \right),
\end{aligned}
\] 
for $n \neq 0$ and:
\[
|\gamma_0(r)| + r |\partial_r \gamma_0(r)|  \leq   
C_0\dfrac{|M_{*}^{ref}|^2}{r^{\beta^*_0} },  \qquad
|w_0(r)| +  r|\partial_r w_0(r)|  
\leq 
C_0\dfrac{|M_{*}^{ref}|^2}{r^{2+\beta^*_0}} .
\]
Then, there is a constant $C_1$ depending only on $C_0$ s.t. the Fourier-modes $F_n$ of ${\nabla^{\bot} \gamma} \cdot \nabla w$ satisfy:
\[
\begin{aligned}
|F_0(r)| & \leq \dfrac{C_1[M^{ref} + |M_{*}^{ref}|^2]^2 }{r^{4 + 2\min(\beta^*_0,\beta_*^{(0)})}}\,, &&\\[6pt]
|F_n (r)| &\leq \dfrac{C_1[M^{ref} + |M_{*}^{ref}|^2]^2}{r^{4 + \min(\beta^*_0+\beta_*^{(0)},2\beta_*^{(0)})}}\,, && \text{ if $|n|=1$}\,, \\[6pt]
|F_n (r)| & \leq \dfrac{C_1[M^{ref} + |M_{*}^{ref}|^2]^2}{r^{4 + 2\beta_*^{(0)}}}, && \text{ if $|n| >1$} \,.
\end{aligned}
\]
\end{lemma}

\begin{proof}
For the proof, we split again:
\[
\begin{aligned}
\left[ \nabla^{\bot} \gamma \cdot \nabla\right] w =& 
\left[ \nabla^{\bot} \gamma_0 \cdot \nabla\right] w_0 + 
\left[ \nabla^{\bot} \gamma_0 \cdot \nabla\right] w_1 + 
\left[ \nabla^{\bot} \gamma_0 \cdot \nabla\right] w^{(1)} +\\
&
\left[ \nabla^{\bot} \gamma_1 \cdot \nabla\right] w_0 + 
\left[ \nabla^{\bot} \gamma_1 \cdot \nabla\right] w_1 + 
\left[ \nabla^{\bot} \gamma_1 \cdot \nabla\right] w^{(1)} +\\
&
\left[ \nabla^{\bot} \gamma^{(1)} \cdot \nabla\right] w_0 + 
\left[ \nabla^{\bot} \gamma^{(1)} \cdot \nabla\right] w_1 +
\left[ \nabla^{\bot} \gamma^{(1)} \cdot \nabla\right] w^{(1)}.  
\end{aligned}
\]

Furthermore, since $(\gamma,w)$ are smooth on $\mathbb R^2 \setminus B,$ we have:
\[
\begin{aligned}
& \nabla^{\bot}{\gamma}^{(0)}(r,\theta) = \sum_{n\neq 0}  \exp(in\theta) \left(- \dfrac{in\gamma_n(r)}{r}{\bf e}_r +  \partial_r \gamma_n(r) {\bf e}_{\theta}\right) ,
\\
&
\nabla {w}^{(0)}(r,\theta) = \sum_{n\neq 0}  \exp(in\theta)
(\partial_r w_n(r) {\bf e}_r + \dfrac{in w_n(r)}{r} {\bf e}_{\theta}).
\end{aligned}\]
In particular, with our assumptions on the decay of $\gamma_n,\, \partial_r \gamma_n$ and $w_n,\, \partial_r w_n,$ together with the introducing remark \eqref{eq_decayboundary}, we infer that:
\begin{equation} \label{eq_decay(0)}
\begin{aligned}
\|\nabla^{\bot} \gamma^{(0)}(r,\cdot)\|_{L^{\infty}(\mathbb T)} & \leq C \dfrac{C_0}{r^{1+ \beta_*^{(0)}}} [M^{ref} + |M_*^{ref}|^2], \\
 \|\nabla w^{(0)}(r,\cdot)\|_{L^{2}(\mathbb T)}
 & \leq C \dfrac{C_0}{r^{3+ \beta_*^{(0)}}} [M^{ref} + |M_*^{ref}|^2].
\end{aligned}
\end{equation}
To compute the zero-mode, we can for instance apply this:
\[
\begin{aligned}
|F_0(r)|  \leq &  \dfrac{1}{\sqrt{2\pi}} \|\nabla^{\bot} \gamma(r,\cdot)\|_{L^{\infty}(T)} \|\nabla w(r,\cdot)\|_{L^2(\mathbb T)} \\
  \leq & 
\dfrac{1}{\sqrt{2\pi}} \left( \|\nabla^{\bot} \gamma_0(r,\cdot)\|_{L^{\infty}(T)} + \|\nabla^{\bot} \gamma^{(0)}(r,\cdot)\|_{L^{\infty}(T)}\right) \\
&
\cdot \left(  \|\nabla w_0(r,\cdot)\|_{L^2(\mathbb T)} +  \|\nabla w^{(0)}(r,\cdot)\|_{L^2(\mathbb T)} \right).
\end{aligned}
\]
Combining  \eqref{eq_decay(0)} with the decay of $\gamma_0, \, w_0$ yields directly the expected result for $F_0$.  When $|n| = 1$ we remark that the term $[\nabla^{\bot} \gamma_0\cdot \nabla ] w_0$ has no contribution and we complement the analysis with the same method. 
When $|n| \geq 2$ we have further that the term:
\[
[\nabla^{\bot} \gamma_0\cdot \nabla ] w_0 + [\nabla^{\bot} \gamma_1\cdot \nabla ] w_0+[\nabla^{\bot} \gamma_0\cdot \nabla ] w_1
\]
does not contribute so that all the involved terms decrease like a non zero mode. We conclude the proof with a similar argument.
\end{proof}

\medskip
 
With these two lemmas at-hand, we complete the proof of {\bf Proposition \ref{prop_decaymode1}}. Indeed, with the assumptions of {\bf Proposition \ref{prop_decaymode1}} we can bound directly:
\[
|F_n(r)| \leq \dfrac{1}{\sqrt{2\pi}}\|{\bf v}(r,\cdot)\|_{L^{\infty}(\mathbb T)} \|\nabla^2 {\bf v}(r,\cdot)\|_{L^{2}(\mathbb T)} \leq \dfrac{1}{\sqrt{2\pi}}\dfrac{|M^{ref}|^2}{r^{4+2\beta_0}},
\] 
for all $n\in \mathbb Z.$ Applying {\bf Lemma \ref{lem_decaynmode}} with $\varepsilon=1/2$ and $\alpha =\beta_0$,  we obtain then that $(\gamma_n,w_n)$ satisfy:
\[
\begin{aligned}
|\gamma_n(r)| + r\dfrac{|\partial_r \gamma_n(r)|}{1+|n|} & \leq   
\dfrac{C}{r^{\min(\phi_0-2,3\beta_0/2)}} \left( |v_{r,n}^*| + |v_{\theta,n}^*| +  \dfrac{|M^{ref}|^2}{(1+|n|)^3} \right), \\
\dfrac{|w_n(r)|}{1+|n|} +  r\dfrac{|\partial_r w_n(r)|}{(1+|n|)^2} &
\leq 
\dfrac{C}{r^{2+ \min(\phi_0-2,3\beta_0/2)}} \left( |v_{r,n}^*| + |v_{\theta,n}^*| +  \dfrac{|M^{ref}|^2}{(1+|n|)^3} \right),
\end{aligned}
\]
for all $n \in \mathbb Z$ (note that $v_{r,0}^* = 0$ and $v_{\theta,0}^* = \mu - {u}_{\theta, 0}^*$ and $\xi_n^{-} < -\phi_0$) so that computations similar to the proof of {\bf Lemma \ref{lem_decayconvol}} yield that:
\[
\|\nabla^{\bot} \gamma(r,\cdot)\|_{L^{\infty}(\mathbb T)} + r^2 \|\nabla w(r,\cdot)\|_{L^{2}(\mathbb T)} \leq \dfrac{C_0}{r^{1+ \min(\phi_0-2,3\beta_0/2)}} (|M^{ref}|+|M^{ref}|^2).
\]
We can then iterate the process and increase little by little the exponent $\beta_0$ as long as $3\beta_0/2 < \phi_0-2.$  In a finite number of steps, we obtain
\[
\begin{aligned}
|\gamma_n(r)| + r\dfrac{|\partial_r \gamma_n(r)|}{1+|n|} & \leq   
\dfrac{C}{r^{(\phi_0-2)}} \left( |v_{r,n}^*| + |v_{\theta,n}^*| +  \dfrac{|{M}_*^{ref}|^2}{(1+|n|)^3} \right), \\
\dfrac{|w_n(r)|}{1+|n|} +  r\dfrac{|\partial_r w_n(r)|}{(1+|n|)^2} &
\leq 
\dfrac{C}{r^{\phi_0}} \left( |v_{r,n}^*| + |v_{\theta,n}^*| +  \dfrac{|{M}_*^{ref}|^2}{(1+|n|)^3} \right),
\end{aligned}
\]
with ${M}_*^{ref}$ of the form $C_0|M^{ref}|^2 + C_1 |M^{ref}|^{2k}$ for some positive constants $C_0, C_1$ and exponents $k$ depending on the number of iterations. 
 
 \medskip
 
We are then in a position to apply {\bf Lemma \ref{lem_decayconvol}}  with $\beta_0^{*} = \beta^{(0)}_* = \phi_0-2$ and $\varepsilon= (\phi_0-2)/(\phi_0-1).$ We obtain that $(F_n)_{n\in \mathbb Z}$ matches the assumptions of {\bf Lemma \ref{lem_decaynmode}} with $\alpha = \phi_0-2,$ in case $n=0$ and $\alpha = (\beta_0^{*} + \beta^{(0)}_*)/2$ in case $n \neq 2.$ 
Independent applications in case $n=0$ or $n\neq 0$ yield no better decay estimate in case $n=0$ but:
\[
\begin{aligned}
|\gamma_n(r)| + r\dfrac{|\partial_r \gamma_n(r)|}{1+|n|} & \leq   
\dfrac{C}{r^{\min(\beta_{1},(2-\varepsilon)(\beta_0^{*} + \beta^{(0)}_*)/2)}} \left( |v_{r,n}^*| + |v_{\theta,n}^*| +  \dfrac{[M^{ref} + M_*^{ref}]^2}{(1+|n|)^3} \right), \\
\dfrac{|w_n(r)|}{1+|n|} +  r\dfrac{|\partial_r w_n(r)|}{(1+|n|)^2} &
\leq 
\dfrac{C}{r^{2+ \min(\beta_1+2,(2-\varepsilon)(\beta_0^{*} + \beta^{(0)}_*)/2)}} \left( |v_{r,n}^*| + |v_{\theta,n}^*| +  \dfrac{[M^{ref} + M_*^{ref}]^2}{(1+|n|)^3} \right),
\end{aligned}
\]
where $(2-\varepsilon)(\beta_0^{*} + \beta^{(0)}_*)/2 =(2-\varepsilon)(\phi_0-2) > \phi_0-2$ (since $\varepsilon < 1$). We obtain that $(\gamma_n,w_n)$ matches the assumptions of {\bf Lemma \ref{lem_decayconvol}}  with $\beta_0^{*} = \phi_0-2$ and  $\beta^{(0)}_* = \min(\beta_1,(2-\varepsilon)(\phi_0-2)).$ We can then iterate the process
as long as $\beta_0^{*} < \beta_1.$ Indeed, whenever $\beta_0^{*} < \beta_1 < 1$ we have:
\[
(2-\varepsilon)\dfrac{(\phi_0-2)+ \beta^{(0)}_*}{2} > \beta^{(0)}_*. 
\]
In a finite number of steps we reach then the value  $\beta^{(0)}_* = \beta_1$ and we obtain the expected result.
We can then iterate once more the convolution/resolution argument. We obtain that the convolution term satisfy:
\[
|F_n (r)|  \leq \dfrac{C_0 M^{ref} + C_1 |M^{ref}|^{2k}]}{r^{4 + 2\beta_1}}, \quad \qquad \forall \, |n| \geq 2,
\]
for some constants $C_1,C_1$ and exponent $k.$ Consequently, we have:
\[
 |\gamma_n(r)|  + \dfrac{r|\partial_r \gamma_n(r)|}{(1+|n|)} \leq  C \left(  |v_{r,n}^*| +  |v_{\theta,_n}^*| + \dfrac{C_0 M^{ref} + C_1 |M^{ref}|^{2k}]}{(1+|n|)^3} \right) r^{\max(2+\xi_2^-,-(2-\varepsilon)\beta_1,-2)},
\]
for all $\varepsilon >0$ arbitrary small. 
Since $\beta_1 > 1/2$ we can choose $\varepsilon$ sufficiently small to infer like previously that:
\[
|{\bf v}^{(1)}| \leq \dfrac{C_0 M^{ref} + C_1 |M^{ref}|^{2k}}{r^{1+2\beta_1^-}}.
\]
The decay information on ${ \nabla {\bf v}_1(x)}$ and $\nabla {\bf v}^{(1)}(x)$ in {\bf Proposition \ref{prop_decaymode1}}  will follow in a similar way.



\end{document}